\documentclass[11pt]{article}
\usepackage[left=2.2cm,right=2.2cm,top=2.5cm,bottom=2.5cm]{geometry}
\usepackage{amsfonts,amssymb,amsthm,amsmath,amscd}
\usepackage{enumerate}

\title{Lower bounds for the magnitude of the minimum eigenvalue of graphs with applications to MaxCut and Chowla's cosine problem}
\author{Fredy Yip\thanks{Trinity College, University of Cambridge, United Kingdom. Email: \textbf{fy276@cam.ac.uk}.}}
\date{}

\usepackage{graphicx}
\usepackage{parskip}

\usepackage{overpic}
\usepackage{url}
\usepackage{color}
\usepackage[dvipsnames]{xcolor}
\usepackage{comment}

\newcommand{\conlem}[2]{{#1}_{\textcolor{NavyBlue}{\ref{#2}}}}
\newcommand{\Clem}[1]{\conlem{C}{#1}}
\newcommand{\clem}[1]{\conlem{c}{#1}}

\newenvironment{custhm}[2]{\trivlist\item[\hskip\labelsep\textbf{#1~#2.}]\em}{\endtrivlist}

\numberwithin{equation}{section}

\newtheorem{thm}{Theorem}[section]
\newtheorem{lem}[thm]{Lemma}
\newtheorem{claim}[thm]{Claim}
\newtheorem{conj}[thm]{Conjecture}
\newtheorem{prob}[thm]{Problem}
\newtheorem{cor}[thm]{Corollary}
\newtheorem{prop}[thm]{Proposition}

\theoremstyle{definition}
\newtheorem{df}[thm]{Definition}

\newtheorem{obs}[thm]{Observation}

\newcommand{\tr}[0]{\operatorname{tr}}
\newcommand{\floor}[1]{\left\lfloor #1\right\rfloor}
\newcommand{\ceil}[1]{\left\lceil #1\right\rceil}
\newcommand{\EE}[1]{\mathbb{E}\left( #1\right)}

\newcommand{\tri}[0]{\ensuremath{{t_3}}}
\newcommand{\surp}[0]{\operatorname{sp}}
\newcommand{\disc}[0]{\operatorname{disc}}
\newcommand{\mc}[0]{\operatorname{mc}}

\begin{document}

\maketitle

\begin{abstract}
    Jin, Milojevi\'c, Tomon and Zhang established a powerful recursive estimate relating the positive eigenvalues of a graph whose least eigenvalue is small in absolute value. We establish a refinement of this result, yielding improved estimates across spectral graph theory and discrepancy theory, and for Chowla's cosine problem. Recent results of Janzer, Tomon and Yip allow us to directly transfer least eigenvalue estimates to the corresponding surplus estimates. 

    Using these tools, we fully resolve a conjecture of R\"aty, Sudakov and Tomon. We show that, for an $n$-vertex graph $G$ with least eigenvalue $\lambda_n$ and surplus $\surp(G)$, if $G$ is $\epsilon$-far from all disjoint unions of cliques, then $|\lambda_n|\geq \Omega_\epsilon(n^{1/4})$ and $\surp(G)\geq \Omega_\epsilon(n^{5/4})$. Furthermore, we show that when $G$ is $n^{-o(1)}$-far from all disjoint unions of cliques, we have $|\lambda_n|\geq n^{1/4 - o(1)}$ and $\surp(G)\geq n^{5/4 - o(1)}$. Finally, we show that the surplus of a $K_t$-free graph with $m$ edges is at least $m^{0.614 - o(1)}$ as $m$ tends to infinity. 
\end{abstract}

\section{Introduction}

Discrepancy theory studies the degrees to which discrete objects may resemble their uniform analogues. An overview of its applications across combinatorics, theoretical computer science and number theory may be found in Chazelle~\cite{disc}. 

In the context of graph theory, discrepancy theory measures the deviation of graphs from uniformly weighted complete graphs. Given an $n$-vertex graph $G$ with average degree $d$, let the density of $G$ be given by $p = \frac{d}{n - 1}$. Erd\H{o}s, Goldberg, Pach, and Spencer~\cite{erdos1988cutting} defined the discrepancy of $G$ as
\begin{equation*}
    \disc(G) = \max_{U\subseteq V(G)} \left|e(G[U]) - p\binom{|U|}{2}\right|, 
\end{equation*}
where $G[U]$ denotes the induced subgraph of $G$ on $U$. 

Crucially, discrepancy offers a measure of the deviation from the uniform object and not from the corresponding random object. In fact, for $p\in [n^{-1}, 1 - n^{-1}]$, the discrepancy of the Erd\H{o}s-R\'enyi random graph $R(n, p)$ is $\Theta(\min(p, 1 - p)^{1/2}n^{3/2})$ with high probability. Following the work of Erd\H{o}s and Spencer~\cite{erdos1972imbalances} on the case where $p = 1/2$, Erd\H{o}s, Goldberg, Pach, and Spencer~\cite{erdos1988cutting} established a matching lower bound. More precisely, they showed that for any $p\in [n^{-1}, 1 - n^{-1}]$ the discrepancy of an $n$-vertex graph with density $p$ is at least $\Omega(\min(p, 1 - p)^{1/2}n^{3/2})$. 

Lower bounds of the discrepancy of a graph $G$ identify induced subgraphs with significantly more or fewer edges than expected from the density of $G$. The concept of signed discrepancy describes the extent to which each of these two outcomes hold. Specifically, the positive and negative discrepancies of a graph $G$ are defined as
\begin{align*}
    \disc^+(G) = \max_{U\subseteq V(G)} \left(e(G[U]) - p\binom{|U|}{2}\right), \\
    \disc^-(G) = \max_{U\subseteq V(G)} \left(p\binom{|U|}{2} - e(G[U])\right), 
\end{align*}
respectively. Note that, upon complementation, we have $\disc^+(G) = \disc^-\left(\overline{G}\right)$. Therefore, it suffices to study one of the two signed discrepancies; we focus on the positive discrepancy. 

Generalising the result of~\cite{erdos1988cutting}, Bollob\'as and Scott~\cite{bollobas2006discrepancy} showed that for an $n$-vertex graph with density $p\in [n^{-1}, 1 - n^{-1}]$, we have $\disc^+(G)\disc^-(G)\geq \Omega(\min(p, 1 - p)n^3)$. This implies the estimate $\disc(G)\geq \Omega(\min(p, 1 - p)^{1/2}n^{3/2})$ of~\cite{erdos1988cutting} since $\disc(G) = \max(\disc^+(G), \disc^-(G))$. When combined with the trivial estimate $\disc^-(G)\leq O(pn^2)$, this yields the lower bound $\disc^+(G)\geq \Omega(n)$ when $p\in [n^{-1}, 1/2]$. 

The positive discrepancy of a graph may be much smaller than its discrepancy. A key example is given by Tur\'an graphs. Specifically, $n$-vertex Tur\'an graphs have positive discrepancy $\Theta(n)$. Since all Tur\'an graphs have density at least $1/2$, one may hope that one may place stronger lower bounds on the positive discrepancy once the density is below and bounded away from $1/2$. Verstra\"ete \cite[Conjecture A]{Verstraete} conjectured that in this regime, when $p\in [n^{-1}, 1/2 - \epsilon]$, we have $\disc^+(G)\geq \Omega_\epsilon(p^{1/2}n^{3/2})$, matching the lower bound for discrepancy. Alon~\cite{alon1997edge} showed that this is indeed the case when $n^{-1}\leq p\leq n^{-8/9}$. R\"aty, Sudakov and Tomon~\cite{rst} extended this estimate to the regime where $n^{-1}\leq p\leq n^{-1/3}$, but showed that this conjecture of Verstra\"ete does not hold in general. Nonetheless, they offered a lower bound for positive discrepancy superior to the $\Omega(n)$ lower bound across the entire regime $n^{-1}\leq p\leq 1/2 - \epsilon$. This lower bound has been strengthened by Janzer, Tomon and Yip~\cite{bridge} into the following form, which is believed~\cite{rst} to be tight. 

\begin{thm}[Combining Theorem~1.2 of~\cite{rst} and Theorem~1.2 of~\cite{bridge}]
    Let $G$ be an $n$-vertex graph with average degree $d$ and let $\epsilon > 0$. Then
    \begin{equation*}
        \disc^+(G) = \begin{cases}
            \Omega(nd^{1/2}) & \text{ if } 1\leq d\leq n^{2/3}\\
            \Omega(n^2/d) & \text{ if } n^{2/3}\leq d\leq n^{3/4}\\
            \Omega_{\epsilon}(nd^{1/3}) & \text{ if } n^{3/4}\leq d\leq (\frac{1}{2}-\epsilon)n.
        \end{cases}
    \end{equation*}
    These bounds are tight for $1\leq d\leq n^{3/4}$.
\end{thm}

For graphs whose densities may exceed $1/2$, R\"aty, Sudakov and Tomon~\cite{rst} conjectured that $\disc^+(G)\geq \Omega(n^{5/4})$ whenever $G$ is far from all Tur\'an graphs (including the empty graph and the complete graph). More precisely, we call graphs $G$ and $H$ on the same set of $n$ vertices \emph{$\epsilon$-far} from one another if the edit distance between them is at least $\epsilon n^2$; we call them \emph{$\epsilon$-close} otherwise. R\"aty, Sudakov and Tomon~\cite{rst} made the following conjecture. 

\begin{conj}[Conjecture~8.1 of~\cite{rst}] \label{5/4 conj}
    Let $G$ be an $n$-vertex graph and let $\epsilon > 0$. If $G$ is $\epsilon$-far from all Tur\'an graphs (including the empty graph and the complete graph), then 
    \begin{equation*}
        \disc^+(G)\geq \Omega_\epsilon(n^{5/4}). 
    \end{equation*}
\end{conj}

The exponent $5/4$ in Conjecture~\ref{5/4 conj} is optimal following a construction of de Caen~\cite{decaen}. Jin, Milojevi\'c, Tomon, and Zhang~\cite{graph} showed that in the setting of Conjecture~\ref{5/4 conj}, we have $\disc^+(G)\geq \Omega_\epsilon(n^{1 + \Omega(1)})$, pushing beyond the trivial bound $\disc^+(G)\geq \Omega_\epsilon(n)$ by an absolute constant in the exponent. Building on results from~\cite{graph}, Janzer, Tomon and Yip~\cite{bridge} asymptotically resolved Conjecture~\ref{5/4 conj} affirmatively. 

\begin{thm}[Theorem~1.3 of~\cite{bridge}] \label{5/4 asymp thm}
    Let $\epsilon > 0$, and let $n$ be sufficiently large with respect to $\epsilon$. If an $n$-vertex graph $G$ is $\epsilon$-far from all Tur\'an graphs (including the empty graph and the complete graph), then 
    \begin{equation*}
        \disc^+(G)\geq n^{5/4 - o_\epsilon(1)}. 
    \end{equation*}
\end{thm}

We fully resolve Conjecture~\ref{5/4 conj} affirmatively, removing the $o(1)$ term in the exponent in Theorem~\ref{5/4 asymp thm}. 

\begin{thm} \label{5/4 thm}
    Let $G$ be an $n$-vertex graph and let $\epsilon > 0$. If $G$ is $\epsilon$-far from all Tur\'an graphs (including the empty graph and the complete graph), then 
    \begin{equation*}
        \disc^+(G)\geq \Omega_\epsilon(n^{5/4}). 
    \end{equation*}
\end{thm}

The proofs of both Theorem~\ref{5/4 asymp thm} and Theorem~\ref{5/4 thm} employ a relationship between the positive discrepancy and eigenvalues of graphs established in~\cite{bridge}. We shall discuss this in more detail in the following subsection. 

We conclude the present discussion of graph discrepancy theory by introducing the MaxCut and surplus of graphs. A \emph{cut} of a graph is a bipartition of its vertex set and the \emph{size} of a cut refers to the number of edges going across this partition. The MaxCut $\mc(G)$ of a graph $G$ refers to the maximum size of a cut in $G$. Taking a random cut where each vertex is assigned to each part with independent probability $1/2$, it follows that $\mc(G)\geq m/2$ in a graph $G$ with $m$ edges. As such, we define the surplus of a graph $G$ with $m$ edges to be $\surp(G):= \mc(G) - m/2$. The surplus of a graph forms a measure of its uniformity akin to its negative discrepancy. In fact, for an $n$-vertex regular graph $G$, we have $\surp(G) = \Theta(\disc^-(G) + n)$, and in general, $\surp(G)\leq O(\disc^-(G) + n)$~\cite{bridge}. Many results giving lower bounds for signed discrepancy have natural extensions which give equivalent lower bounds for surplus under similar assumptions. We shall include such extensions as appropriate. For instance, we have the following variant of Theorem~\ref{5/4 thm} which addresses the surplus instead. 

\begin{thm} \label{5/4 maxcut thm}
    Let $G$ be an $n$-vertex graph and let $\epsilon > 0$. If $G$ is $\epsilon$-far from all disjoint unions of cliques, then 
    \begin{equation*}
        \surp(G)\geq \Omega_\epsilon(n^{5/4}). 
    \end{equation*}
\end{thm}

Theorems~\ref{5/4 thm} and~\ref{5/4 maxcut thm} are related by complementation. We shall in fact deduce Theorem~\ref{5/4 thm} from Theorem~\ref{5/4 maxcut thm}. Note the condition of Theorem~\ref{5/4 maxcut thm} concerns all disjoint unions of cliques without requiring these cliques to be of similar order. Indeed, whilst complete multipartite graphs may have large positive discrepancy if they are far from Tur\'an graphs, all disjoint unions of cliques have small surplus. 

\subsection{Eigenvalues}

Given an $n$-vertex graph $G$, let its eigenvalues $\lambda_1\geq \dots \geq \lambda_n$ refer to the eigenvalues of its adjacency matrix. The study of eigenvalues of graphs (particularly $\lambda_2$ and $\lambda_n$) has long been central in analysing the pseudorandomness and expansion of graphs. In the context of discrepancy theory, they have been useful algebraic proxies of positive discrepancy and MaxCut which are often more tractable than their combinatorial analogues. In particular, for a $d$-regular graph, R\"aty, Sudakov and Tomon~\cite{rst} noted that
\begin{equation*}
    \disc^+(G)\leq n\lambda_2/2 + d. 
\end{equation*}
Similarly, as noted in \cite{alon1996bipartite, delorme1993laplacian, mohar1990eigenvalues}, we have
\begin{equation*}
    \surp(G)\leq n|\lambda_n|/4. 
\end{equation*}
Unfortunately, lower bounds of similar efficacy for $\surp(G)$ in terms of $|\lambda_n|$ are unavailable. Indeed, taking $G$ to be the disjoint union of a clique $K_{n/2}$ and a star $K_{1, n/2 - 1}$, we have $\surp(G) = \Theta(n)$ whilst $|\lambda_n| = \Theta(n^{1/2})$. Nevertheless, strategies for giving lower bounds for $\lambda_2$ and $|\lambda_n|$ have greatly informed approaches in establishing lower bounds for the positive discrepancy and the surplus, respectively, albeit often leading to weaker bounds in the combinatorial setting. 

Recently, Janzer, Tomon and Yip~\cite{bridge} showed that this passage from eigenvalue estimates to lower bounds for measures of discrepancy is in fact direct and loseless once suitable eigenvalue estimates are given for all moderately-sized induced subgraphs. We state a version of their result relating the surplus and $|\lambda_n|$. 

\begin{thm}[Theorem~1.8 of~\cite{bridge}] \label{sp bridge thm}
    Let $G$ be a graph on $n$ vertices, and let $n'\leq n$ be a positive integer. Assume that for any induced subgraph $H$ of $G$ on at least $n'$ vertices, the absolute value of the least eigenvalue of $H$ is at least~$\lambda$. Then $\surp(G)\geq (n - n')\lambda/4$. 
\end{thm}

An application of Theorem~\ref{sp bridge thm} will reduce Theorem~\ref{5/4 maxcut thm} to the following estimate on $|\lambda_n|$, which will be proven in Section~\ref{dense graph sec}. This reduction argument will also be given in Section~\ref{dense graph sec}. 

\begin{thm} \label{1/4 thm}
    Let $G$ be an $n$-vertex graph and let $\epsilon > 0$. Let $\lambda_n$ denote the least eigenvalue of $G$. If $G$ is $\epsilon$-far from all disjoint unions of cliques, then 
    \begin{equation*}
        |\lambda_n|\geq \Omega_\epsilon(n^{1/4}). 
    \end{equation*}
\end{thm}

A similar but weaker result has been established in~\cite{graph}, with an estimate of the form $|\lambda_n|\geq n^{1/4 - o(1)}$. 

\begin{thm}[Theorem~1.4 of~\cite{graph}] \label{1/4 - o(1) thm}
    Let $\epsilon > 0$, and let $n$ be sufficiently large with respect to $\epsilon$. If an $n$-vertex graph $G$ is $\epsilon$-far from all disjoint unions of cliques, then the least eigenvalue $\lambda_n$ of $G$ satisfies 
    \begin{equation*}
        |\lambda_n|\geq n^{1/4 - o_\epsilon(1)}. 
    \end{equation*}
\end{thm}

The proof of Theorem~\ref{1/4 - o(1) thm} in~\cite{graph} follows from an eigenvalue recursion estimate (Lemma~6.3 of~\cite{graph}) they establish, which lies at the heart of their argument with various applications. This estimate offers significant control over the positive eigenvalues of $G$ when $|\lambda_n|$ is small. We give below the relevant quantity considered by this estimate. 

\begin{df}
    For a non-negative real number $T$ and an $n$-vertex graph $G$ with eigenvalues $\lambda_1\geq\dots\geq \lambda_n$, let 
    \begin{equation*}
        S_T = \sum_{i\in [n]: \lambda_i \geq T}\lambda_i
    \end{equation*}
    denote the sum of the eigenvalues of $G$ (counted with multiplicity) which are at least $T$. 
\end{df}

\begin{lem}[Lemma~6.3 of~\cite{graph}] \label{key quad}
    There exists an absolute constant $C\geq 1$ such that the following holds. Let $G$ be an $n$-vertex graph with least eigenvalue $\lambda_n$. If $T$ and $T'$ are non-negative real numbers satisfying
    \begin{equation} \label{key quad equation 1}
        \frac{T'}{Cn} = \left(\frac{T}{Cn}\right)^2, 
    \end{equation}
    and $T'\geq |\lambda_n|^2$, then 
    \begin{equation} \label{key quad equation 2}
        \frac{S_{T'}}{Cn}\geq \left(\frac{S_T}{Cn}\right)^2. 
    \end{equation}
\end{lem}

Jin, Milojevi\'c, Tomon, and Zhang~\cite{graph} noted that an iterative application of Lemma~\ref{key quad} implies that Lemma~\ref{key quad} holds with any power of two in place of two in the exponents in~(\ref{key quad equation 1}) and~(\ref{key quad equation 2}). Our main technical result shows that this holds in fact for any positive integer. 

\begin{prop} \label{key}
    There exists an absolute constant $C\geq 1$ such that the following holds. Let $k\geq 1$ be a positive integer. If $T$ and $T'$ are non-negative real numbers satisfying
    \begin{equation*}
        \frac{T'}{Cn} = \left(\frac{T}{Cn}\right)^k, 
    \end{equation*}
    and $T'\geq |\lambda_n|^2$, then 
    \begin{equation*}
        \frac{S_{T'}}{Cn}\geq \left(\frac{S_T}{Cn}\right)^k. 
    \end{equation*}
\end{prop}

It is important to note that a straightforward extension of the argument used in~\cite{graph} to establish Lemma~\ref{key quad} would yield a version of Proposition~\ref{key} which requires $T'\geq |\lambda_n|^k$. In some scenarios, this would render the result completely ineffective for $k\geq 4$. Therefore, the key strength of Proposition~\ref{key} lies in the fact that the weaker assumption $T'\geq |\lambda_n|^2$ suffices. 

After our proofs of Theorems~\ref{5/4 thm} and~\ref{5/4 maxcut thm} have been announced in~\cite{bridge}, the author was informed by Xu and Zhou of their independent proofs of Theorems~\ref{5/4 thm} and~\ref{1/4 thm}. Their approach employed the aforementioned weaker variant of Proposition~\ref{key} with $k = 3$, which indeed suffices for this application. 

Jin, Milojevi\'c, Tomon, and Zhang~\cite{graph} showed that when $|\lambda_n|\leq n^\gamma$ for any fixed positive $\gamma < 1/6$, the graph is $n^{-\alpha}$-close (polynomially close) to a disjoint union of cliques, for some $\alpha = \alpha(\gamma) > 0$. 

\begin{thm}[Theorem~1.5 of~\cite{graph}] \label{1/6 poly thm}
    For any $\gamma\in (0, 1/6)$, there exists $\alpha > 0$ such that the following holds when $n$ is sufficiently large with respect to $\gamma$. If the least eigenvalue $\lambda_n$ of an $n$-vertex graph $G$ satisfies $|\lambda_n|\leq n^\gamma$, then $G$ is $n^{-\alpha}$-close to a disjoint union of cliques. 
\end{thm}

The example of de Caen~\cite{decaen} shows that this cannot hold for $\gamma\geq 1/4$. It is natural to wonder if Theorem~\ref{1/6 poly thm} continues to hold whenever $\gamma < 1/4$. A numerical optimisation of the approach of~\cite{graph} may yield Theorem~\ref{1/6 poly thm} for $\gamma < 1/5$, but not beyond the exponent $1/5$. Using Proposition~\ref{key} in place of Lemma~\ref{key quad}, we show that Theorem~\ref{1/6 poly thm} indeed holds for all $\gamma < 1/4$. 

\begin{thm}\label{1/4 poly thm}
    For any $\gamma\in (0, 1/4)$, there exists $\alpha > 0$ such that the following holds when $n$ is sufficiently large with respect to $\gamma$. If the least eigenvalue $\lambda_n$ of an $n$-vertex graph $G$ satisfies $|\lambda_n|\leq n^\gamma$, then $G$ is $n^{-\alpha}$-close to a disjoint union of cliques. 
\end{thm}

Theorem~\ref{1/4 poly thm} will be proven in Section~\ref{poly sec}. Theorem~\ref{sp bridge thm} yields the following surplus version of Theorem~\ref{1/4 poly thm}. 

\begin{thm}\label{5/4 poly maxcut thm}
    For any $\gamma\in (0, 5/4)$, there exists $\alpha > 0$ such that the following holds when $n$ is sufficiently large with respect to $\gamma$. If $\surp(G)\leq n^\gamma$, then $G$ is $n^{-\alpha}$-close to a disjoint union of cliques. 
\end{thm}

We omit the deduction of Theorem~\ref{5/4 poly maxcut thm} from Theorem~\ref{1/4 poly thm} as it is analogous to the deduction of Theorem~\ref{5/4 maxcut thm} from Theorem~\ref{1/4 thm}. 

The use of Proposition~\ref{key} in place of Lemma~\ref{key quad} also allows for significant quantitative improvements in a range of other applications, including MaxCut in $K_t$-free graphs and Chowla's cosine problem. We finish our introduction with a discussion of these results. 

\subsection{MaxCut of $K_t$-free graphs}

A classical result of Edwards~\cite{edwards1973some, edwards1975improved} showed that the surplus of any graph with $m$ edges is at least $\Omega(m^{1/2})$, which is attained by cliques. A natural question is to investigate whether stronger surplus lower bounds may be obtained when the graph is, in an appropriate sense, far from being a clique. Our discussion above (\emph{cf.}~Theorem~\ref{5/4 maxcut thm}) explored such lower bounds for graphs which are $\epsilon$-far from disjoint unions of cliques. Approaching this problem from a different angle, the study of the MaxCut and surplus of $H$-free graphs, where $H$ is a fixed small graph, was initiated by Erd\H{o}s and Lov\'asz~\cite{erdHos1975problems} for triangle-free graphs, and has since been studied by~\cite{alon1996bipartite, alon2003maximum, alon2005maxcut, balla2024maxcut, carlson2021lower, glock2023new, graph, shearer1992note}. 

We focus on the most pertinent case where $H = K_t$ is a small clique. In a sense, this is general, since any $H$-free graph is necessarily $K_{v(H)}$-free. In the context of triangle-free graphs, Alon~\cite{alon1996bipartite} showed that the positive discrepancy of a triangle-free graph with $m$ edges is at least $\Omega(m^{4/5})$, which is tight. More generally, Alon, Bollob\'as, Krivelevich and Sudakov~\cite{alon2003maximum} showed that the positive discrepancy of a $K_t$-free graph with $m$ edges is at least $\Omega_t(m^{1/2 + \epsilon_t})$, for some $\epsilon_t > 0$ tending to zero as $t$ tends to infinity. Motivated by random graphs of appropriate densities, they conjectured that the positive discrepancy of a $K_t$-free graph with $m$ edges is at least $\Omega_t(m^{3/4 + \epsilon_t'})$, for some $\epsilon_t > 0$ depending on $t$. A first progress towards this conjecture was given by Jin, Milojevi\'c, Tomon, and Zhang~\cite{graph}, who showed that the positive discrepancy of a $K_t$-free graph with $m$ edges is at least $\Omega_t(m^{1/2 + \epsilon})$ for some $\epsilon > 0$ which is independent of $t$. More precisely, they showed that $\surp(G)\geq \Omega_t(m^{0.51})$. Leveraging Theorem~\ref{sp bridge thm} and the spectral results of~\cite{graph}, \cite{bridge} improved this lower bound, showing that $\surp(G)\geq \Omega_t(m^{0.55})$ for any $K_t$-free graph $G$ with $m$ edges. 

As noted in~\cite{bridge}, this exponent of $0.55$ may be further improved with a numerical optimisation. However, such an optimisation based on the methods of~\cite{graph} and~\cite{bridge} hits a fundamental obstacle at an exponent 0.6. We show that, using Proposition~\ref{key} in place of Lemma~\ref{key quad}, we may improve this exponent to $0.614 - o(1)$. The analogous fundamental obstacle for our method sits at $5/8 = 0.625$, roughly corresponding to the construction of de Caen~\cite{decaen}. In fact, the stronger exponent $0.616$ follows from our methods after numerical optimisation. For the sake of clarity, we present our argument for the exponent $0.614 - o(1)$. 

\begin{thm} \label{Kt maxcut thm}
    Let $t$ be a fixed positive integer. If $G$ is a $K_t$-free graph with $m$ edges, then
    \begin{equation*}
        \surp(G)\geq m^{0.614 - o_t(1)}
    \end{equation*}
    as $m$ tends to infinity. 
\end{thm}

To show Theorem~\ref{Kt maxcut thm}, we make use of the following reduction result from~\cite{bridge}, which reduces the problem to the corresponding estimate on $|\lambda_n|$. 

\begin{thm}[Proposition~6.1 of~\cite{bridge}] \label{maxcut m bridge thm}
    Let $\alpha\leq 1$ be a positive real number, and let $t$ be a positive integer. Assume that every $K_t$-free $n$-vertex graph $G$ with $m$ edges and least eigenvalue $\lambda_n$ satisfies $n|\lambda_n|\geq \Omega_t(m^\alpha)$. Then every $K_t$-free graph $G$ with $m$ edges satisfies $\surp(G)\geq m^{\alpha - o_t(1)}$. 
\end{thm}

Theorem~\ref{maxcut m bridge thm} reduces Theorem~\ref{Kt maxcut thm} to the following estimate of $|\lambda_n|$, which will be proven in Section~\ref{maxcut Kt sec}. 

\begin{thm} \label{Kt lambda n thm}
    Let $t$ be a fixed positive integer. If $G$ is a $K_t$-free graph with $m$ edges and least eigenvalue $\lambda_n$, then
    \begin{equation*}
        n|\lambda_n|\geq \Omega_t(m^{0.614}). 
    \end{equation*}
\end{thm}

\subsection{Chowla's cosine problem}

Finally we discuss an application to Chowla's celebrated cosine problem on the minimum value of the Fourier form of a set of $n$ positive integers. 

Given a (finitely-supported) function $f: \mathbb{Z}\rightarrow \mathbb{C}$, let its Fourier transform $\hat{f}:\mathbb{R}/\mathbb{Z}\rightarrow \mathbb{C}$ be given by
\begin{equation*}
    \hat{f}(k) = \sum_{x\in \mathbb{Z}}e^{-2\pi i kx}f(x). 
\end{equation*}
A central theme in understanding properties of the Fourier transform in additive combinatorics is to understand how structural properties of $f$ relate to those of its Fourier transform $\hat f$. In particular, in combinatorial settings, the function $f$ often arises as the indicator function $1_X$ of a subset $X\subseteq \mathbb{Z}$. It is natural to investigate the restrictions that this assumption places upon $\hat{f}$. 

Perhaps the most influential problem under this theme is Littlewood's $L^1$-conjecture~\cite{littlewoodl1} on the minimal $L^1$ norm of the Fourier transform $\hat{f}$ of an indicator function $f = 1_X$, where $X\subseteq \mathbb{Z}$ has prescribed size. 

\begin{conj}[Littlewood's $L^1$-conjecture] \label{littlewood l1}
    Let $X\subseteq \mathbb{Z}$ be a set of $n$ integers. Then for $f = 1_X$, we have
    \begin{equation*}
        \|\hat{f}\|_1 := \int_0^1|\hat f| \gg \log n, 
    \end{equation*}
    where the implied constant is absolute. 
\end{conj}

The lower bound $\Omega(\log n)$ is attained when $X$ is an $n$-term arithmetic progression. Progress towards Conjecture~\ref{littlewood l1} began with Cohen's~\cite{cohen} work giving the first lower bound on $\|\hat{f}\|_1$ which tended to infinity with $n$. Following decades of iterative improvement by Davenport~\cite{Davenport}, Pichorides~\cite{Pichorides1974, pichorides1977, pichorides1978, pichorides1980}, and Fournier~\cite{Fournier}, Conjecture~\ref{littlewood l1} was finally resolved positively in 1981 independently by Konjagin~\cite{Konjagin} and McGehee, Pigno, and Smith~\cite{mcgeheepignosmith}. 

\begin{thm}[Konjagin~\cite{Konjagin}; McGehee, Pigno, and Smith~\cite{mcgeheepignosmith}] \label{littlewood l1 thm}
    Let $X\subseteq \mathbb{Z}$ be a set of $n$ integers. Then for $f = 1_X$, we have
    \begin{equation*}
        \|\hat{f}\|_1 \gg \log n. 
    \end{equation*}
\end{thm}

Another notable problem under this theme of investigating analytical properties of Fourier transforms of indicator functions is one of Chowla~\cite{Chowla}. This problem, stated as Problem~\ref{chowla problem} below, has come to be known as Chowla's cosine problem and culminates in Conjecture~\ref{chowla}, which remains open. 

Let $X$ be a set of $n$ positive integers, and let $f = 1_{X\cup -X}$. Chowla asked for lower bounds on $-\min \hat{f}$. 

\begin{df}
    For a positive integer $n$, let $\mathfrak{m}(n)$ denote the infimum of
    \begin{equation*}
        -\min \widehat{1_{X\cup -X}} = -2\min_{k\in \mathbb{R}/\mathbb{Z}}\left(\sum_{x\in X}\cos(2\pi kx)\right)
    \end{equation*}
    over sets $X$ of $n$ positive integers. 
\end{df}

\begin{prob}[Chowla's cosine problem] \label{chowla problem}
    Determine the asymptotic growth of $\mathfrak{m}(n)$ as $n$ tends to infinity. 
\end{prob}

\begin{conj}[Chowla~\cite{Chowla}] \label{chowla}
    We have $\mathfrak{m}(n)\gg \sqrt{n}$. 
\end{conj}

Note that $\int \hat{f} = 0$ for $f = 1_{X\cup -X}$. Therefore, $-\min \hat{f}\geq 0$ and thus $\mathfrak{m}(n)$ is always non-negative. The conjectured lower bound $\mathfrak{m}(n)\gg \sqrt{n}$ may be motivated by noting that $\int |\hat{f}|^2 = 2n$, and therefore the standard deviation of $\hat{f}$ from 0 is on the order of $\sqrt{n}$. From this perspective, Conjecture~\ref{chowla} asserts that the contribution to this standard deviation does not originate entirely from the positive side, and that this probabilistic heuristic remains valid. 

The conjectured lower bound $\mathfrak{m}(n)\gg \sqrt{n}$ is attained when $X = (S - S)\cap \mathbb{Z}^+$ for a Sidon set $S$. Indeed, in this case
\begin{equation*}
    f := 1_{X\cup -X} = 1_{(S - S)\backslash \{0\}} = 1_S * 1_{-S} - |S|\cdot 1_{\{0\}}. 
\end{equation*}
Therefore, 
\begin{equation*}
    \hat{f} = |\widehat{1_S}|^2 - |S|, 
\end{equation*}
hence $-\min \hat{f}\leq |S| = \Theta(\sqrt{n})$. 

Chowla's motivations for Conjecture~\ref{chowla} stem from his work~\cite{chowlaant} in analytic number theory. However, Chowla's cosine problem is also intimately connected to Littlewood's $L^1$-conjecture. To see this, note that for an even real function $f = 1_{X\cup -X}$, the Fourier transform $\hat{f}$ is also even and real. Therefore, since $\int \hat{f} = 0$, we have
\begin{equation*}
    -\min \hat{f} \geq \int_0^11_{\hat{f}\leq 0}|\hat{f}| = \frac{1}{2}\|\hat{f}\|_1. 
\end{equation*}
Hence, up to absolute constants, any lower bound for $\|\hat{f}\|_1$ translates into one for $-\min \hat{f}$. In particular, Theorem~\ref{littlewood l1 thm} implies that
\begin{equation*}
    \mathfrak{m}(n)\gg \log n. 
\end{equation*}
Pushing beyond this logarithmic bound from Littlewood's $L^1$-conjecture, Bourgain~\cite{bourgain} gave a quasi-polynomial lower bound for $\mathfrak{m}(n)$, which Ruzsa~\cite{ruzsa} refined to
\begin{equation*}
    \mathfrak{m}(n)\geq \exp(\Omega((\log n)^{1/2})). 
\end{equation*}

Bourgain~\cite{bourgain} and Ruzsa~\cite{ruzsa} employed methods of additive combinatorics and harmonic analysis. Under the assumption of a small $-\min \hat{f}$, they identified structure within the set $X\subseteq \mathbb{Z}^+$, finding polynomial-many translates of an image of a hypercube of poly-logarithmic dimension. Ruzsa noted that, in particular, $X$ must contain the sumset of two sets of quasi-polynomial size (and a third with two elements), from which a lower bound on $-\min \hat{f}$ was proven. 

The lower bound of Ruzsa remained the state of the art for many years until two breakthrough works by Jin, Milojevi\'c, Tomon, and Zhang~\cite{graph} and Bedert~\cite{bedert} in 2025, appearing in rapid succession of one another, both giving a polynomial lower bound $\mathfrak{m}(n)\geq n^{\Omega(1)}$. The approaches taken by~\cite{graph} and~\cite{bedert} are different. Jin, Milojevi\'c, Tomon, and Zhang~\cite{graph} employed methods of spectral graph theory, as discussed earlier, whilst Bedert~\cite{bedert} employed additive combinatorial and harmonic analytical methods, extending and utilising results from the earlier works of Bourgain and Ruzsa. 

The approach of Jin, Milojevi\'c, Tomon, and Zhang~\cite{graph} originally yielded an unspecified exponent which was subsequently improved to $1/10$. 

\begin{thm}[Jin, Milojevi\'c, Tomon, and Zhang~\cite{graph}] \label{JMTZ graph chowla thm}
    We have
    \begin{equation*}
        \mathfrak{m}(n)\geq n^{1/10 - o(1)}. 
    \end{equation*}
\end{thm}

This exponent of $1/10$ may be improved slightly by numerically optimising their argument. However, such optimisations cannot improve this exponent beyond $1/8$. Employing Proposition~\ref{key} in place of Lemma~\ref{key quad}, we improve this exponent to $1/6.43$. 

\begin{thm} \label{chowla main}
    We have
    \begin{equation*}
        \mathfrak{m}(n)\geq \Omega(n^{1/6.43}). 
    \end{equation*}
\end{thm}

The proofs of Theorem~\ref{chowla main} (albeit with a weaker exponent $1/6.6$ for clarity of presentation) and Proposition~\ref{key} were the content of a part III essay titled ``Chowla's cosine problem" set by Timothy Gowers as a part of the Cambridge Mathematical Tripos. At the time of the official submission of the essay in May 2026, the best existing exponent on Chowla's cosine problem was $1/7$, due to Bedert~\cite{bedert}. However, this exponent has since been improved to $1/5$ by Bedert in July 2026. 

\textbf{Organisation of the paper.} The rest of the paper is organised as follows. Section~\ref{key sec} establishes our key eigenvalue recursion estimate, Proposition~\ref{key}. Section~\ref{spec dist sec} derives estimates on the distribution of the positive eigenvalues on the basis of Proposition~\ref{key}. Section~\ref{dense graph sec} establishes results for dense graphs, including Theorems~\ref{5/4 thm},~\ref{5/4 maxcut thm} and~\ref{1/4 thm}, as well as results used in the rest of the paper. Section~\ref{tri counting sec} establishes a high triangle count on the basis of the spectral distribution estimates of Section~\ref{spec dist sec}. This is used to show the existence of dense neighbourhoods in Section~\ref{dense nbh sec}. Section~\ref{iter densi sec} describes an iterative densification procedure by repeated passages to dense neighbourhoods. Section~\ref{rec sec} analyses the recursion arising from this process. Sections~\ref{maxcut Kt sec} and~\ref{Chowla cosine sec} address the applications of these results in the settings of MaxCut in $K_t$-free graphs (proving Theorems~\ref{Kt maxcut thm} and~\ref{Kt lambda n thm}) and Chowla's cosine problem (proving Theorem~\ref{chowla main}). And finally, Section~\ref{poly sec} gives a proof of Theorem~\ref{1/4 poly thm}. 

\textbf{General notation.} For a positive integer $n$, let $[n] := \{1, \dots, n\}$ denote the set of positive integers smaller than or equal to $n$. We make use of the Vinogradov notation $f\ll g$ to denote $f = O(g)$, where the implied constant will always be absolute unless stated or denoted otherwise. We omit floor and ceiling functions where they do not play a crucial role. 

\textbf{Graph-theoretic notation.} Unless otherwise specified, a graph $G$ refers to a finite simple graph. We use $V(G)$ and $E(G)$ to denote the set of vertices and the set of edges of $G$, respectively. For a subset $S\subseteq V(G)$ of vertices, let $E(S)$ denote the set of edges whose endpoints both lie in $S$. For disjoint subsets $S_1, S_2\subseteq V(G)$ of vertices, let $E(S_1, S_2)$ denote the set of edges between $S_1$ and $S_2$. We denote the cardinality of these sets by the corresponding lower case letter. That is, $v(G) = |V(G)|$, $e(G) = |E(G)|$, $e(S) = |E(S)|$ and $e(S_1, S_2) = |E(S_1, S_2)|$. For a vertex $i\in V(G)$, let $\deg i = \deg_G i$ denote its degree in $G$. Let $d = d(G) = 2e(G)/v(G)$ denote the average degree of $G$, and let $p = p(G) = d(G)/(v(G) - 1)$ denote its density. We use $H\subseteq G$ to denote that $H$ is a subgraph of $G$, and we use $H\sqsubseteq G$ to denote that $H$ is an induced subgraph of $G$. 

\section{A key spectral inequality} \label{key sec}

\begin{df}
    For two $n\times n$ symmetric (real) matrices $X, Y$, we use $X\succeq Y$ to denote that $X - Y$ is positive-semi-definite. 
\end{df}

The property which distinguishes adjacency matrices of graphs from arbitrary symmetric matrices is that adjacency matrices are Boolean. To capture this property, we consider the entry-wise product of matrices of equal dimensions, often known as the Hadamard product. Under this definition, adjacency matrices are precisely symmetric matrices which are equal to their Hadamard product with themselves. 

\begin{df}
    For two matrices $X, Y$ of equal dimensions, let the \emph{Hadamard product} $X\circ Y$ denote the matrix obtained from $X$ and $Y$ via entry-wise multiplication. In other words $(X\circ Y)_{ij} = X_{ij}Y_{ij}$. For a positive integer $k$, let $X^{\circ k}$ denote the $k$-fold Hadamard product of $X$ with itself. Let $J$ denote the all-ones matrix, the multiplicative identity under the Hadamard product. 
    
    Let 
    \begin{equation*}
        \langle X, Y\rangle := \sum_{ij}X_{ij}Y_{ij} = \tr (X^TY)
    \end{equation*}
    be the entry-wise inner product of $X$ and $Y$. Let the \emph{Frobenius norm} 
    \begin{equation*}
        \|X\|_F := \langle X, X\rangle^{1/2}
    \end{equation*}
    be the corresponding norm. 
\end{df}

Note that when applied to vectors $v, w$ of equal length, $\langle v, w\rangle$ reduces to the inner product $v\cdot w = v^Tw$ and the Frobenius norm reduces to the standard $\ell_2$ length of vectors. 

Perhaps somewhat surprisingly, the entry-wise notion of the Hadamard product is related to the global notion of positive-semi-definiteness. More specifically, the Hadamard product of positive-semi-definite matrices is a positive-semi-definite matrix. This fact, sometimes known as the Schur product theorem~\cite{Schur}, has far-reaching consequences despite its simple proof. 

\begin{lem}[Schur product theorem]
    Let $X\succeq 0$, $Y\succeq 0$ be positive-semi-definite $n\times n$ matrices, then $X\circ Y\succeq 0$. 
\end{lem}

\begin{proof}
    Diagonalising, we may write $X = \sum_{i = 1}^n a_i u_iu_i^T$, $Y = \sum_{i = 1}^n b_i w_iw_i^T$, where $a_i, b_i\geq 0$. Therefore, 
    \begin{equation*}
        X\circ Y = \sum_{i, j}a_ib_j (u_i\circ w_j)(u_i\circ w_j)^T\succeq 0. \qedhere
    \end{equation*}
\end{proof}

Not so surprisingly, we note that the Hadamard product distributes over matrix summation. This gives us the following corollary of the Schur product theorem. 

\begin{cor} \label{schur cor}
    Let $X\succeq X'\succeq 0$, $Y\succeq Y'\succeq 0$ be positive-semi-definite $n\times n$ matrices, then $X\circ Y\succeq X'\circ Y'$. 
\end{cor}

We similarly define the Hadamard product of vector subspaces of $\mathbb{R}^n$. 

\begin{df}
    For subspaces $U, W\leq \mathbb{R}^n$, let their \emph{Hadamard product} 
    \begin{equation*}
        U\circ W := \operatorname{span}\left(\{u\circ w|u\in U, w\in W\}\right)\leq \mathbb{R}^n
    \end{equation*}
    be the subspace spanned by Hadamard products of vectors in $U$ with vectors in $W$. For a positive integer $k$, let $U^{\circ k}$ denote the $k$-fold Hadamard product of $U$ with itself. 
\end{df}

\begin{lem} \label{dim bound}
    For subspaces $U, W\leq \mathbb{R}^n$, we have $\dim(U\circ W)\leq \dim U \dim W$. 
\end{lem}

\begin{proof}
    Let $k = \dim U, \ell = \dim W$. It suffices to show that for any basis $u_1, \dots, u_k$ of $U$ and $w_1, \dots, w_\ell$ of $W$, the $k\ell$ vectors $u_i\circ w_j$ span $U\circ W$. This follows since for any $u\in U$, $w\in W$, the vector $u\circ w$ lies in the span of the vectors $u_i\circ w_j$. 
\end{proof}

\begin{df}
    For a subspace $V\leq \mathbb{R}^n$, let $P_V$ denote the corresponding projection matrix. For a symmetric $n\times n$ matrix $X$, let the \textit{restriction} of $X$ to $V$ be $X_V = P_V XP_V$. Let $\tr _V(X) = \tr (X_V)$ denote the trace of $X_V$. We say that $X$ is \emph{supported on} $V$ if $X = X_V$. 
\end{df}

$X_V$ is defined so that the restriction of the bilinear form corresponding to $X$ to the subspace $V$ gives the bilinear form corresponding to $X_V$. 

The restriction of a positive-semi-definite matrix must also be positive-semi-definite. Therefore, if $X\succeq Y$, then for any $V\leq \mathbb{R}^n$, we have $X_V\succeq Y_V$ and $\tr_V(X)\geq \tr_V(Y)$. 

The following lemma relates the Hadamard product of matrices with the Hadamard product of subspaces. 

\begin{lem} \label{support Hadamard}
    If $X$ is supported on $U$ and $Y$ is supported on $W$, then $X\circ Y$ is supported on $U\circ W$. 
\end{lem}

\begin{proof}
    We may write $X = X_U$ as a finite linear combination $\sum_i \alpha_i u_iu_i^T$ of vectors $u_i\in U$, and $Y = Y_W$ as a finite linear combination $\sum_j \beta_j w_jw_j^T$ of vectors $w_j\in W$. Therefore, 
    \begin{equation*}
        X\circ Y = \sum_{i, j} \alpha_i\beta_j (u_i\circ w_j)(u_i\circ w_j)^T
    \end{equation*}
    is supported on $U\circ W$. 
\end{proof}

\begin{df}
    For $T\geq 0$, let $V_T\leq \mathbb{R}^n$ denote the subspace spanned by the eigenvectors $v_i$ with an eigenvalue $\lambda_i\geq T$. Let $N_T = \dim V_T$ denote the number of eigenvalues (counted with multiplicity) greater than or equal to $T$. Let
    \begin{equation*}
        A_{\geq T} = A_{\geq T}(G) := A_{V_T} = \sum_{\lambda_i\geq T} \lambda_iv_iv_i^T, 
    \end{equation*}
    where we sum over $i\in [n]$ for which $\lambda_i\geq T$. Let 
    \begin{equation*}
        S_T = S_T(G) := \sum_{\lambda_i\geq T}\lambda_i = \tr A_{\geq T}
    \end{equation*}
    be its trace. 
\end{df}

Note that we have the decomposition
\begin{equation*}
    A = \sum_{i = 1}^n \lambda_iv_iv_i^T. 
\end{equation*}

Note that $S_T\geq TN_T$ by Markov's inequality. Both $S_T$ and $N_T$ are (not necessarily strictly) decreasing functions of $T$. 

An upper bound of $S_T$ may be given by the trace identity
\begin{equation} \label{trace id 1}
    \sum_{i = 1}^n\lambda_i = \tr A = 0. 
\end{equation}
This may be considered the ``trivial'' upper bound on $S_T$. Improving this bound will lie at the heart of our argument. 

\begin{lem} \label{triv S_T bound}
    For any $T\geq 0$, we have
    \begin{equation*}
        S_T\leq n|\lambda_n|. 
    \end{equation*}
\end{lem}

\begin{proof}
    By Equation~(\ref{trace id 1}), we have
    \begin{equation*}
        S_T\leq \sum_{\lambda_i\geq 0}\lambda_i = - \sum_{\lambda_i < 0}\lambda_i\leq n|\lambda_n|. 
    \end{equation*}
    Here, and in what follows, we abuse notation and use $\sum_{\lambda_i\geq 0}$, for example, to denote the sum over $i\in [n]$ for which $\lambda_i \geq 0$. 
\end{proof}

The following Lemma bounds the trace of $A$ over a low-dimensional subspace $V$. 

\begin{lem} \label{markov ish}
    For any subspace $V\leq \mathbb{R}^n$ and any $T\geq 0$, we have
    \begin{equation*}
        \tr_V(A)\leq S_T + T\dim V. 
    \end{equation*}
\end{lem}

\begin{proof}
    Since $A\preceq A_{\geq T} + TI$, we have
    \begin{equation*}
        \tr_V(A)\leq \tr_V(A_{\geq T}) + \tr_V(TI)\leq \tr(A_{\geq T}) + T\dim V = S_T + T\dim V. \qedhere
    \end{equation*}
\end{proof}

As discussed earlier, the key property of an adjacency matrix $A$ is that $A = A^{\circ 2}$. In hopes of applying the Schur product theorem, we consider the positive-semi-definite matrix $A + |\lambda_n|I\succeq 0$. 

\begin{df}
    Let the \emph{modified adjacency matrix} be $\tilde{A} := A + |\lambda_n|I\succeq 0$
\end{df}

Since $A\circ I = 0$, we have
\begin{equation} \label{mod adj sq eq}
    \tilde{A}^{\circ 2} = A + |\lambda_n|^2 I\preceq \tilde{A} + |\lambda_n|^2 I. 
\end{equation}
As a corollary of Lemma~\ref{markov ish}, we may also give an upper bound on the trace of $A$ over a low-dimensional subspace $V$. 

\begin{cor} \label{markov ish cor}
    For any subspace $V\leq \mathbb{R}^n$ and any $T\geq 0$, we have
    \begin{equation*}
        \tr_V(\tilde{A})\leq S_T + (T + |\lambda_n|)\dim V. 
    \end{equation*}
\end{cor}

The following recursive spectral inequality, stated earlier in the introduction as Proposition~\ref{key}, lies at the heart of our argument. Heuristically, under the assumption that negative eigenvalues are close to zero, Proposition~\ref{key} dictates the distribution of the mass of the positive eigenvalues, which are much larger in absolute value than the negative eigenvalues. It generalises Lemma~\ref{key quad} (Lemma~6.3 of~\cite{graph}), which corresponds to the special case $k = 2$ of Proposition~\ref{key}. 

\begin{custhm}{Proposition}{\ref{key}}
    There exists an absolute constant $C = \Clem{key}\geq 1$ such that the following holds. Let $k\geq 1$ be a positive integer. If $T$ and $T'$ are non-negative real numbers satisfying
    \begin{equation*}
        \frac{T'}{Cn} = \left(\frac{T}{Cn}\right)^k, 
    \end{equation*}
    and $T'\geq |\lambda_n|^2$, then 
    \begin{equation} \label{key prop statement ineq}
        \frac{S_{T'}}{Cn}\geq \left(\frac{S_T}{Cn}\right)^k. 
    \end{equation}
\end{custhm}

Before we give the proof of Proposition~\ref{key} in its full generality, we first give a proof for the special case $k = 2$, corresponding to Lemma~\ref{key quad}. 

\begin{proof}[Proof in the case where $k = 2$; \emph{i.e.}, a proof of Lemma~\ref{key quad}]
    Fix $C = 2^{10}$. Let $L = |\lambda_n|$. Since $\tilde{A}\succeq \tilde{A}_{V_T}\succeq 0$, it follows from Corollary~\ref{schur cor} that
    \begin{equation*}
        \tilde{A} + L^2I\succeq \tilde{A}^{\circ 2}\succeq \tilde{A}_{V_T}^{\circ 2}. 
    \end{equation*}
    Taking the projection onto $V_T^{\circ 2}$ gives
    \begin{equation} \label{k = 2 eq}
        \tilde{A}_{V_T^{\circ 2}} + L^2P_{V_T^{\circ 2}}\succeq \tilde{A}_{V_T}^{\circ 2}. 
    \end{equation}
    Taking the trace, and applying the power-mean inequality, we have
    \begin{equation*}
        \tr \tilde{A}_{V_T^{\circ 2}} + L^2\dim V_T^{\circ 2}\geq \sum_{i = 1}^n ((\tilde{A}_{V_T})_{ii})^2\geq n^{-1}\left(\tr \tilde{A}_{V_T}\right)^2\geq n^{-1}S_T^2. 
    \end{equation*}
    However, by Corollary~\ref{markov ish cor}, we have
    \begin{equation*}
        \tr \tilde{A}_{V_T^{\circ 2}}\leq S_{T'} + (T' + L)\dim V_T^{\circ 2}. 
    \end{equation*}
    Hence
    \begin{equation*}
        S_{T'}\geq n^{-1}S_T^2 - (T' + L^2 + L)\dim V_T^{\circ 2}. 
    \end{equation*}
    Lemma~\ref{dim bound} gives 
    \begin{equation*}
        \dim V_T^{\circ 2}\leq (\dim V_T)^2 = N_T^2\leq T^{-2}S_T^2. 
    \end{equation*}
    Therefore, 
    \begin{equation*}
        S_{T'}\geq n^{-1}S_T^2\left(1 - \frac{T' + L^2 + L}{CT'}\right)\geq \frac{S_T^2}{Cn}. \qedhere
    \end{equation*}
\end{proof}

Iterating the special case $k = 2$ of Proposition~\ref{key} implies the special case of Proposition~\ref{key} where $k$ is a power of two. This is the approach taken in~\cite{graph}. The large gap between powers of two led to a relatively coarse bound on the contribution of small positive eigenvalues to the quadratic mass. The main role of Proposition~\ref{key} is to drastically tighten this estimate, which eventually leads to a stronger bound on the triangle count (see Section~\ref{tri counting sec}) and a more rapid densification process (see Section~\ref{iter densi sec}). 

One may attempt to generalise the proof above for $k = 2$ to general values of $k$ by establishing a $k$-fold generalisation 
\begin{equation*}
    \tilde{A}_{V_T^{\circ k}} + |\lambda_n|^kP_{V_T^{\circ k}}\succeq A_{V_T}^{\circ k}. 
\end{equation*}
of~(\ref{k = 2 eq}). This line of approach does indeed establish an inequality of the form~(\ref{key prop statement ineq}). However, it requires the much stronger assumption that $T'\geq |\lambda_n|^k$ instead of the assumption $T'\geq |\lambda_n|^2$ in Proposition~\ref{key}, severely limiting its applicability in later parts of our argument. We now give the proof of Proposition~\ref{key} in its full generality. 

\begin{proof}[Proof of Proposition~\ref{key} for general values of $k$]
    Take $C = 2^{10}$. We may assume that $T\leq n$, otherwise $S_T = 0$ and the desired inequality holds trivially. Similarly, since the desired inequality is trivial for $k = 1$, we may assume that $k\geq 2$. Let $L = |\lambda_n|$. The desired inequality holds trivially when $e(G) = 0$. We assume that $e(G) > 0$. Hence $L\geq 1$. 
    
    Let $U = V_T$, $A' = \tilde{A}_U = A_{\geq T} + LP_{V_T}$ and $N = N_T = \dim U$. Note that $\tr A' = S_T + LN$.  By Lemma~\ref{dim bound}, $\dim (U^{\circ \ell})\leq \dim(U)^\ell = N_T^\ell$ for any positive integer $\ell$, a fact that we shall use throughout. 
    
    For each $\ell\leq k$, we seek to find a matrix $M = M_\ell$ satisfying $0\preceq M\preceq \tilde{A}$ with rank at most $N^\ell$ and high trace. In particular, it would suffice for each diagonal entry $M_{ii}$ of $M$ to be comparable to $(A'_{ii})^{\ell}$. Formally, we employ the following definitions. 

    For a positive integer $\ell$, $i\in [n]$, and a matrix $M$ with rank at most $N^\ell$ satisfying $0\preceq M\preceq \tilde{A}$, let 
    \begin{equation*}
        c_{\ell, i}(M) := A'_{ii} - (M_{ii})^{1/\ell}\leq A'_{ii}
    \end{equation*}
    and let
    \begin{equation*}
        c_\ell(M) := \sum_{i = 1}^n \max(c_{\ell, i}(M), 0). 
    \end{equation*}
    Let 
    \begin{equation*}
        c_\ell = \min_M c_\ell(M), 
    \end{equation*}
    where we take the minimum over matrices $M$ with rank at most $N^\ell$ satisfying $0\preceq M\preceq \tilde{A}$. 
    
    We shall give an upper bound on $c_\ell$ via induction on $\ell$. We first establish a suitable recursive inequality on $c_\ell$. 

    \begin{claim} \label{key claim}
        For any positive integers $\ell_1, \ell_2$. Let $\ell = \ell_1 + \ell_2$. Then
        \begin{equation*}
            c_\ell\leq c_{\ell_1} + c_{\ell_2} + n^{1 - 1/\ell}L^{2/\ell}N. 
        \end{equation*}
    \end{claim}

    \begin{proof}
        Take $M_1$ and $M_2$ such that $0\preceq M_1, M_2\preceq \tilde{A}$, the rank of $M_1$ is at most $N^{\ell_1}$, the rank of $M_2$ is at most $N^{\ell_2}$, and 
        \begin{align*}
            c_{\ell_1}(M_1) &= c_{\ell_1}, \\
            c_{\ell_2}(M_2) &= c_{\ell_2}. 
        \end{align*}
        Define $M' = M_1\circ M_2\succeq 0$. Note that the rank of $M'$ is at most $N^\ell$ and by Corollary~\ref{schur cor}, we have
        \begin{equation*}
            M'\preceq \tilde{A}^{\circ 2}\preceq \tilde{A} + L^2I. 
        \end{equation*}
        To obtain a matrix $M$ from $M'$ which satisfies $M\preceq \tilde{A}$, we consider the following quasi-projection of $M$. Let $P = \tilde{A}^{1/2}(\tilde{A} + L^2I)^{-1/2}$. Note that $0\preceq P\preceq I$. Take $M = PM'P$, the rank of $M$ is at most $N^\ell$ and we have
        \begin{equation*}
            0\preceq M\preceq \tilde{A}. 
        \end{equation*}
        Hence $c_\ell\leq c_\ell(M)$. We now relate the diagonal entries of $M$ to those of $M'$. Let $E = (I - P)M'(1 - P)$ account for the error induced by this quasi-projection. Note that the rank of $E$ is also at most $N^\ell$ and
        \begin{equation*}
            E\preceq (I - P)(\tilde{A} + L^2I)(I - P) = \left(\sqrt{\tilde{A} + L^2I} - \sqrt{\tilde{A}}\right)^2\preceq L^2I. 
        \end{equation*}
        Applying the triangle inequality to the Euclidean norm of the $i$th row of the matrix identity
        \begin{equation*}
            M'^{1/2} = PM'^{1/2} + (I - P)M'^{1/2}, 
        \end{equation*}
        it follows that
        \begin{equation*}
            \sqrt{M'_{ii}}\leq \sqrt{M_{ii}} + \sqrt{E_{ii}}. 
        \end{equation*}
        Since $(x + y)^\alpha\leq x^\alpha + y^\alpha$ for any $x, y\geq 0$ and $\alpha\leq 1$, we have
        \begin{equation*}
            (M_{ii}')^{1/\ell}\leq (M_{ii})^{1/\ell} + (E_{ii})^{1/\ell}
        \end{equation*}
        Therefore, 
        \begin{align*}
            (M_{ii})^{1/\ell}&\geq (M_{ii}')^{1/\ell} - (E_{ii})^{1/\ell} = ((M_1)_{ii})^{1/\ell}((M_2)_{ii})^{1/\ell} - (E_{ii})^{1/\ell}\\
            &\geq (A'_{ii} - c_{\ell_1, i}(M_1))^{\ell_1/\ell}(A'_{ii} - c_{\ell_2, i}(M_2))^{\ell_2/\ell} - (E_{ii})^{1/\ell}\\
            &\geq A'_{ii} - c_{\ell_1, i}(M_1) - c_{\ell_2, i}(M_2) - (E_{ii})^{1/\ell}
        \end{align*}
        Therefore, 
        \begin{equation*}
            c_{\ell, i}(M)\leq c_{\ell_1, i}(M_1) + c_{\ell_2, i}(M_2) + (E_{ii})^{1/\ell}. 
        \end{equation*}
        Hence, 
        \begin{equation*}
            c_\ell\leq c_{\ell}(M)\leq c_{\ell_1} + c_{\ell_2} + \sum_{i = 1}^n(E_{ii})^{1/\ell}. 
        \end{equation*}
        Since the rank of $E$ is at most $N^\ell$ and $0\preceq E\preceq L^2I$, we have $\tr E\leq L^2 N^\ell$. Hence, by the power-mean inequality, 
        \begin{equation*}
            \sum_{i = 1}^n(E_{ii})^{1/\ell}\leq n^{1 - 1/\ell}(\tr E)^{1/\ell}\leq n^{1 - 1/\ell} L^{2/\ell} N, 
        \end{equation*}
        as needed, establishing the claim. 
    \end{proof}
    
    We return to the proof of Proposition~\ref{key}. We shall prove via induction on $\ell$ that for any positive integer $\ell\leq k$, 
    \begin{equation} \label{key ind}
        c_\ell\leq 2n^{1 - 1/\ell}L^{2/\ell}N. 
    \end{equation}
    The base case $\ell = 1$ follows since $c_1\leq c_1(A') = 0$. 

    We now seek to establish~(\ref{key ind}) for a given $\ell\in [2, k]$, given that~(\ref{key ind}) holds for all smaller values. Applying Claim~\ref{key claim} with $\ell_1 = \floor{ \frac{\ell}{2}}$ and $\ell_2 = \ceil{\frac{\ell}{2}}$, we have
    \begin{equation*}
        c_\ell\leq c_{\ell_1} + c_{\ell_2} + n^{1 - 1/\ell}L^{2/\ell}N. 
    \end{equation*}
    Note that the positive integers $\ell_1, \ell_2$ are both strictly smaller than $\ell$. Therefore, by our inductive hypothesis, we have
    \begin{align*}
        c_{\ell_1}&\leq 2n^{1 - 1/\ell_1}L^{2/\ell_1}N, \\
        c_{\ell_2}&\leq 2n^{1 - 1/\ell_2}L^{2/\ell_2}N. 
    \end{align*}
    Therefore, 
    \begin{align*}
        c_\ell&\leq c_{\ell_1} + c_{\ell_2} + n^{1 - 1/\ell}L^{2/\ell}N\\
        &\leq 2n^{1 - 1/\ell_1}L^{2/\ell_1}N + 2n^{1 - 1/\ell_2}L^{2/\ell_2}N + n^{1 - 1/\ell}L^{2/\ell}N\\
        &\leq 4n^{1 - 1/\ell_2}L^{2/\ell_2}N + n^{1 - 1/\ell}L^{2/\ell}N. 
    \end{align*}
    To establish~(\ref{key ind}), it suffices now to show that 
    \begin{equation*}
        4n^{1 - 1/\ell_2}L^{2/\ell_2}N\leq n^{1 - 1/\ell}L^{2/\ell}N. 
    \end{equation*}
    This amounts to showing that
    \begin{equation*}
        \frac{n}{L^2}\geq 4^{\frac{\ell\ell_2}{\ell - \ell_2}} = 4^{\frac{\ell\ell_2}{\ell_1}}. 
    \end{equation*}
    However, note that 
    \begin{equation*}
        \frac{L^2}{Cn}\leq \frac{T'}{Cn} = \left(\frac{T}{Cn}\right)^k\leq C^{-k}. 
    \end{equation*}
    Therefore, 
    \begin{equation*}
        \frac{n}{|\lambda_n|^2}\geq C^{k - 1}\geq 4^{2k}\geq 4^{\frac{\ell\ell_2}{\ell_1}}, 
    \end{equation*}
    since $\ell\leq k$ and $\ell_2\leq 2\ell_1$. This concludes our inductive proof of~(\ref{key ind}) for positive integers $\ell\leq k$. 

    In particular, taking $\ell = k$, we have
    \begin{equation*}
        c_k\leq 2n^{1 - 1/k}L^{2/k}N. 
    \end{equation*}
    Take $M$ such that $0\preceq M\preceq \tilde{A}$, the rank of $M$ is at most $N^k$ and $c_k(M) = c_k$. By the power-mean inequality, we have
    \begin{equation*}
        \tr A' = \sum_{i = 1}^n A'_{ii} \leq \left(\sum_{i = 1}^n(M_{ii})^{1/k}\right) + c_k \leq n^{1 - 1/k}\left(\tr M\right)^{1/k} + c_k. 
    \end{equation*}
    Therefore, 
    \begin{equation*}
        \tr M\geq \left(n^{-1 + 1/k}S_T - 2L^{2/k}N\right)^k\geq \left(\frac{15}{16}n^{-1 + 1/k}S_T\right)^k\geq n^{1 - k}\left(\frac{15}{16}S_T\right)^k, 
    \end{equation*}
    since $n^{-1 + 1/k}S_T\geq n^{-1 + 1/k}TN = C^{1 - 1/k}T'^{1/k}N\geq C^{1 - 1/k}L^{2/k}N\geq 32L^{2/k}N$. Let $V$ denote the support of $M$. It follows that $\tr M\leq \tr_V(\tilde{A})$. Hence, by Corollary~\ref{markov ish cor}, we have
    \begin{equation*}
        \tr M\leq S_{T'} + (T' + L)N^k\leq S_{T'} + (T' + L)N^k\leq S_{T'} + 2T'N^k. 
    \end{equation*}
    Therefore, 
    \begin{align*}
        S_{T'}&\geq n^{1 - k}\left(\frac{15}{16}S_T\right)^k - 2T'N^k\\
        &\geq n^{1 - k}S_T^k\left(\left(\frac{15}{16}\right)^k - 2\frac{n^{k - 1}T'}{T^k}\right)\\
        &\geq n^{1 - k}S_T^k\left(\left(\frac{15}{16}\right)^k - 2\cdot C^{1 - k}\right)\\
        &\geq (Cn)^{1 - k}S_T^k,  
    \end{align*}
    as needed, concluding the proof of Proposition~\ref{key}. 
\end{proof}

Proposition~\ref{key} allows us to derive a stronger upper bound on $S_T$ than Lemma~\ref{triv S_T bound} by applying Lemma~\ref{triv S_T bound} to $S_{T'}$. 

In our inductive derivation of the bound~(\ref{key ind}) on $c_\ell$ in Proposition~\ref{key}, it is crucial that we apply Claim~\ref{key claim} with $\ell_1, \ell_2\leq (1 - \Omega(1))\ell$. The more na\"ive choice, $\ell_1 = 1, \ell_2 = \ell - 1$ (or vice versa), would lead to an extra factor of $k$ in the upper bound for $c_k$, ultimately preventing the densification to a subgraph of constant density in Section~\ref{iter densi sec}. 

\section{Spectral distribution} \label{spec dist sec}

The upper bounds on $S_t$ provided by Proposition~\ref{key} will be used in this section to estimate $\sum_{0\leq \lambda_i\leq T}\lambda_i^2$ via the inequality
\begin{equation} \label{Fubini}
    \sum_{0\leq \lambda_i\leq T}\lambda_i^2\leq \sum_{\lambda_i\geq 0}\lambda_i\min(\lambda_i, T) = \int_0^TS_tdt. 
\end{equation}
The equality in~(\ref{Fubini}) follows from an application of Fubini's theorem. This estimate on $\sum_{0\leq \lambda_i\leq T}\lambda_i^2$ will allow us to show that the main contribution to quadratic energy comes from relatively large eigenvalues. The precise form of this estimate is given by the following proposition. 

\begin{prop} \label{spectral distribution prop}
    Let $r\geq 2$ be a positive integer, let $T, L > 1$ be positive real numbers, and let $G$ be an $n$-vertex graph with least eigenvalue $\lambda_n$. We assume that  $|\lambda_n|\leq L$, $n\geq L^4$, and $T\geq n^{1/2}L$. Then
    \begin{equation*}
        \sum_{\lambda_i\leq T}\lambda_i^2\leq \Clem{spectral distribution prop}\left( nL^{\frac{1}{r}}T + n^{3/2}L^2 + n^{2 - \frac{1}{r}}L^{\frac{2}{r} + \frac{1}{r - 1}}\right), 
    \end{equation*}
    where $\Clem{spectral distribution prop}\geq 1$ is an absolute constant. 
\end{prop}

An asymptotically optimal choice of $r$ in Proposition~\ref{spectral distribution prop} is given by $r = \floor{\frac{\log(n/T)}{\log(n/L^2)}}$. However, we shall prove Proposition~\ref{spectral distribution prop} for general values of $r\geq 2$. 

The rest of this section is dedicated to the proof of Proposition~\ref{spectral distribution prop}. To this end, we fix a positive integer $r$, positive real numbers $T, L$, and an $n$-vertex graph $G$ satisfying the assumptions of Proposition~\ref{spectral distribution prop}. Since
\begin{equation*}
    \sum_{\lambda_i < 0}\lambda_i^2\leq n|\lambda_n|^2\leq nL^2\ll n^{3/2}L^2,  
\end{equation*}
it suffices to show that
\begin{equation} \label{spectral dist ineq form 1}
    \sum_{0\leq \lambda_i\leq T}\lambda_i^2\ll nL^{\frac{1}{r}}T + n^{3/2}L^2 + n^{2 - \frac{1}{r}}L^{\frac{2}{r} + \frac{1}{r - 1}}. 
\end{equation}

\begin{lem} \label{almost convexity lem}
    We have
    \begin{equation*}
        n^{3/2}L^2 + n^{2 - \frac{1}{r}}L^{\frac{2}{r} + \frac{1}{r - 1}}\asymp\max_{k\in \mathbb{Z}, 2\leq k\leq r}n^{2 - \frac{1}{k}}L^{\frac{2}{k} + \frac{1}{k - 1}}. 
    \end{equation*}
\end{lem}

\begin{proof}
    It suffices to show that 
    \begin{equation} \label{almost convexity ineq 1}
        n^{2 - \frac{1}{k}}L^{\frac{2}{k} + \frac{1}{k - 1}}\leq \max(n^{3/2}L^2, n^{2 - \frac{1}{r}}L^{\frac{2}{r} + \frac{1}{r - 1}})
    \end{equation}
    for any $k\in [2, r]$. Let $\gamma = \log_nL$. To show (\ref{almost convexity ineq 1}), it suffices to show that the smooth function
    \begin{equation*}
        x\rightarrow f_1(x) = - \frac{1}{x} + \gamma\left(\frac{2}{x} + \frac{1}{x - 1}\right)
    \end{equation*}
    attains its maximum over $[2, r]$ at an endpoint. If not, let $f(x) > \max(f(2), f(r))$ for some $x_1\in (2, r)$. Therefore, by the mean value theorem, there exist $x_2\in (2, x_1)$ and $x_3\in (x_1, r)$ such that $f_1'(x_2) > 0 > f_1'(x_3)$. In particular, $x_2^2f_1'(x_2) > 0 > x_3^2f_1'(x_3)$, which contradicts the fact that $x^2f_1'(x) = 1 - 2\gamma - \gamma\left(1 + \frac{1}{x - 1}\right)^2$ is an increasing function of $x$ over $[2, \infty)$. 
\end{proof}

Lemma~\ref{almost convexity lem} and~(\ref{Fubini}) together reduce~(\ref{spectral dist ineq form 1}) to showing that
\begin{equation} \label{spectral dist ineq form 2}
    \int_0^TS_tdt\ll nL^{\frac{1}{r}}T + \max_{k\in \mathbb{Z}, 2\leq k\leq r}n^{2 - \frac{1}{k}}L^{\frac{2}{k} + \frac{1}{k - 1}}. 
\end{equation}

We now employ the results of Section~\ref{key sec} to estimate $S_t$ on the left-hand side of~(\ref{spectral dist ineq form 2}). By Lemma~\ref{triv S_T bound}, for any $t > 0$, we have $S_t\leq n|\lambda_n|\leq nL$. In particular, $S_{L^2}\leq n|\lambda_n|\leq nL$. Applying Proposition~\ref{key} with $T' = L^2$, it follows that for any positive integer $k$, we have
\begin{equation*}
    S_{T_k'}\leq \Clem{key}n(L/\Clem{key})^{1/k}\ll nL^{1/k}, 
\end{equation*}
where $T_k' := \Clem{key}n (L^2/\Clem{key}n)^{1/k}$. Taking $T_k := n^{1 - 1/k}L^{2/k}\geq T_k'/\Clem{key}$, we have
\begin{equation*}
    S_{\Clem{key}T_k}\leq S_{T_k'}\ll nL^{1/k}. 
\end{equation*}
Note that $T_k$ is an increasing sequence of positive real numbers. 

For the sake of technical convenience, we shall estimate $\int_0^T S_{\Clem{key}t}dt$ instead of $\int_0^T S_tdt$. This suffices as
\begin{equation*}
    \int_0^T S_tdt\leq \int_0^{\Clem{key}T} S_tdt = \Clem{key} \int_0^T S_{\Clem{key}t}dt. 
\end{equation*}
In other words, to show Proposition~\ref{spectral distribution prop}, it suffices to show that
\begin{equation} \label{spectral dist ineq form 3}
    \int_0^TS_{\Clem{key}t}dt\ll nL^{\frac{1}{r}}T + \max_{k\in \mathbb{Z}, 2\leq k\leq r}n^{2 - \frac{1}{k}}L^{\frac{2}{k} + \frac{1}{k - 1}}. 
\end{equation}
Note that
\begin{equation*}
    \int_0^{T_2}S_{\Clem{key}t}dt\leq T_2 nL = n^{3/2}L^2\leq \max_{k\in \mathbb{Z}, 2\leq k\leq r}n^{2 - \frac{1}{k}}L^{\frac{2}{k} + \frac{1}{k - 1}}, 
\end{equation*}
and, if $T > T_r$, 
\begin{equation*}
    \int_{T_r}^{T}S_{\Clem{key}t}dt\leq (T - T_r)S_{\Clem{key}T_r}\ll nL^{1/r}T. 
\end{equation*}
Therefore, (\ref{spectral dist ineq form 3}) further reduces to showing that
\begin{equation} \label{spectral dist ineq form 4}
    \int_{T_2}^{T_r}S_{\Clem{key}t}dt\ll \max_{k\in \mathbb{Z}, 2\leq k\leq r}n^{2 - \frac{1}{k}}L^{\frac{2}{k} + \frac{1}{k - 1}}. 
\end{equation}

For any positive integer $r_0\geq 2$, we note that
\begin{equation} \label{small r equation}
    \int_{T_2}^{T_{r_0}}S_{\Clem{key}t}dt = \sum_{k = 3}^{r_0} \int_{T_{k - 1}}^{T_k}S_{\Clem{key}t}dt\ll \sum_{k = 3}^{r_0} T_kS_{\Clem{key}T_{k - 1}} = \sum_{k = 3}^{r_0} n^{2 - \frac{1}{k}}L^{\frac{2}{k} + \frac{1}{k - 1}}\leq r_0\max_{k\in \mathbb{Z}, 2\leq k\leq r}n^{2 - \frac{1}{k}}L^{\frac{2}{k} + \frac{1}{k - 1}}. 
\end{equation}
We fix $r_0 = 10$. Note that~(\ref{spectral dist ineq form 4}) follows from~(\ref{small r equation}) if $r\leq r_0$. Therefore, we may assume that $r > r_0$. In this case, (\ref{small r equation}) reduces~(\ref{spectral dist ineq form 4}) to showing that
\begin{equation} \label{spectral dist ineq form 5}
    \int_{T_{r_0}}^{T_r}S_{\Clem{key}t}dt\ll \max_{k\in \mathbb{Z}, 2\leq k\leq r}n^{2 - \frac{1}{k}}L^{\frac{2}{k} + \frac{1}{k - 1}}. 
\end{equation}
We shall in fact show that
\begin{equation} \label{spectral dist ineq form 6}
    \int_{T_{r_0}}^{T_r}S_{\Clem{key}t}dt\ll n^{2 - \frac{1}{r}}L^{\frac{2}{r} + \frac{1}{r - 1}}. 
\end{equation}

We now estimate the left-hand side of~(\ref{spectral dist ineq form 6}). Fixing $\tau = 0.1$, we have
\begin{align*}
    \int_{T_{r_0}}^{T_r}S_{\Clem{key}t}dt &\leq \left(\sup_{t\in (T_{r_0}, T_r)}S_{\Clem{key}t}t^{1 - \tau}\right) \cdot \int_{T_{r_0}}^{T_r} t^{\tau - 1}dt\ll \sup_{t\in (T_{r_0}, T_r)}S_{\Clem{key}t}t^{1 - \tau}T_r^{\tau}\\
    &= 
    \max_{r_0 < k\leq r}\sup_{t\in (T_{k - 1}, T_k)}S_{\Clem{key}t}t^{1 - \tau}T_r^{\tau}\leq \max_{r_0 < k\leq r}S_{\Clem{key}T_{k - 1}}T_k^{1 - \tau}T_r^{\tau}\\
    &\ll \max_{r_0 < k\leq r}n^{2 - \frac{1}{k}}L^{\frac{2}{k} + \frac{1}{k - 1}}(n/L^2)^{\frac{\tau}{k} - \frac{\tau}{r}}. 
\end{align*}
Therefore, to show~(\ref{spectral dist ineq form 6}), it suffices to show that the maximum above is attained when $k = r$. This is the content of the following lemma, which concludes the proof of Proposition~\ref{spectral distribution prop}. 

\begin{lem}
    For any positive integer $k\in (r_0, r]$, we have
    \begin{equation*}
        n^{2 - \frac{1}{k}}L^{\frac{2}{k} + \frac{1}{k - 1}}(n/L^2)^{\frac{\tau}{k} - \frac{\tau}{r}}\leq n^{2 - \frac{1}{r}}L^{\frac{2}{r} + \frac{1}{r - 1}}. 
    \end{equation*}
\end{lem}

\begin{proof}
    Let $\gamma = \log_n L\in (0, 1/4]$. The asserted inequality is equivalent to 
    \begin{equation*}
        \left(\frac{1}{k - 1} - \frac{1}{r - 1}\right)\leq (1 - \tau)(\gamma^{-1} - 2)\left(\frac{1}{k} - \frac{1}{r}\right), 
    \end{equation*}
    which may further be written as
    \begin{equation*}
        \frac{1}{(k - 1)(r - 1)}\leq (1 - \tau)(\gamma^{-1} - 2)\frac{1}{kr}. 
    \end{equation*}
    In other words, we seek to show that
    \begin{equation*}
        \left(1 + \frac{1}{k - 1}\right)\left(1 + \frac{1}{r - 1}\right)\leq (1 - \tau)(\gamma^{-1} - 2). 
    \end{equation*}
    This follows since $r - 1\geq k - 1\geq r_0 = 10$. 
\end{proof}

\section{Results for dense graphs} \label{dense graph sec}

We first give a proof of Theorem~\ref{1/4 thm}, which we restate in the following equivalent form. 

\begin{custhm}{Theorem}{\ref{1/4 thm}'}
    For any $\epsilon > 0$, there exists $c = \clem{1/4 thm}(\epsilon) > 0$ such that the following holds. Let an $n$-vertex graph $G$ be $\epsilon$-far from any disjoint union of cliques, then the least eigenvalue $\lambda_n$ of $G$ satisfies
    \begin{equation*}
        |\lambda_n|\geq cn^{1/4}. 
    \end{equation*}
\end{custhm}

\begin{lem} \label{dense lem}
    For any $\delta > 0$, there exists $c = \clem{dense lem}(\delta) > 0$ such that the following holds. Let $G$ be an $n$-vertex graph $G$ with least eigenvalue $\lambda_n$. If $|\lambda_n|\leq cn^{1/4}$, then 
    \begin{equation*}
        \sum_{\lambda_i\leq cn}\lambda_i^2\leq \delta n^2. 
    \end{equation*}
\end{lem}

\begin{proof}
    Without loss of generality, we assume that $\delta < 0.1$. We may assume also that $n$ is sufficiently large with respect to $\delta$. Otherwise, we may take $c$ to be sufficiently small such that
    \begin{equation*}
        \sum_{\lambda_i\leq cn}\lambda_i^2\leq c^2n^3\leq \delta n^2. 
    \end{equation*}
    We take $c\in (0, 1)$ to be fixed later. Applying Proposition~\ref{spectral distribution prop} with $T = cn$ and $L = cn^{1/4}$, we have, for any positive integer $r\geq 2$, 
    \begin{equation*}
        \sum_{\lambda_i\leq cn}\lambda_i^2\leq \Clem{spectral distribution prop}\left(c^{1 + \frac{1}{r}}n^{2 + \frac{1}{4r}} + c^2n^2 + c^{\frac{2}{r} + \frac{1}{r - 1}}n^{2 - \frac{1}{2r} + \frac{1}{4(r - 1)}}\right)\leq 2c\Clem{spectral distribution prop}n^{2 + \frac{1}{4r}} + \Clem{spectral distribution prop}n^{2 - \frac{1}{2r} + \frac{1}{4(r - 1)}}. 
    \end{equation*}
    When $r\geq 10$, we have
    \begin{equation*}
        \sum_{\lambda_i\leq cn}\lambda_i^2\leq 2c\Clem{spectral distribution prop}n^{2 + \frac{1}{4r}} + \Clem{spectral distribution prop}n^{2 - \frac{1}{5r}}. 
    \end{equation*}
    Take $r = \floor{\frac{\log n}{\log (2\Clem{spectral distribution prop}/\delta)}}/5$. Since $n$ is sufficiently large with respect to $\delta$, we may assume that $r\geq 10$ and $r\geq \frac{\log n}{8\log (2\Clem{spectral distribution prop}/\delta)}$. Therefore, 
    \begin{equation*}
        \sum_{\lambda_i\leq cn}\lambda_i^2\leq 2c\Clem{spectral distribution prop}n^{2 + \frac{1}{4r}} + \Clem{spectral distribution prop}n^{2 - \frac{1}{5r}}\leq 2c\Clem{spectral distribution prop}\left(2\Clem{spectral distribution prop}/\delta\right)^2n^2 + \frac{1}{2}\delta n^2. 
    \end{equation*}
    Taking $c\leq \delta^3\Clem{spectral distribution prop}^{-3}/16$ yields the desired inequality. 
\end{proof}

To establish Theorem~\ref{1/4 thm}' from Lemma~\ref{dense lem}, we employ the following result, which is a direct consequence of Lemmas~8.2 and 8.3 of~\cite{graph}, and the induced graph removal lemma of Alon, Fischer, Krivelevich, and Szegedy~\cite{inducedrem}. For the sake of completeness, we sketch the proof below (see~\cite{graph}, Section 8 for a definition of a $\delta$-regular partition into $K$ parts). 

\begin{prop} \label{dense regpartition reduction prop}
    For any $\epsilon > 0$, there exists $\delta = \conlem{\delta}{dense regpartition reduction prop}(\epsilon) > 0$ such that the following holds. Let $G$ be an $n$-vertex graph with least eigenvalue $\lambda_n$. If $G$ is $\epsilon$-far from any disjoint union of cliques, and for some $c > 0$, 
    \begin{equation*}
        \sum_{\lambda_i\leq cn}\lambda_i^2\leq \delta n^2, 
    \end{equation*}
    then $|\lambda_n|\gg_{c, \epsilon}n$. 
\end{prop}

\begin{proof}
    Note that a disjoint union of cliques is precisely an induced-$P_2$-free graph, where $P_2$ denotes the path of two edges. Since $G$ is $\epsilon$-far from disjoint unions of cliques, by the induced graph removal lemma, there exists $\eta = \eta(\epsilon) > 0$ such that there are at least $\eta n^3$ induced copies of $P_2$ in $G$. Without loss of generality, we assume that $\eta < 0.1$. 

    Let $v_1, \dots, v_n$ be an orthonormal set of eigenvectors corresponding to the eigenvalues $\lambda_1\geq \dots\geq \lambda_n$. Let $B = \sum_{\lambda_i > cn}\lambda_i v_iv_i^T$. The rank of $B$ is bounded above by
    \begin{equation*}
        |\{i\mid\lambda_i > cn\}|\leq \frac{\sum_{i = 1}^n\lambda_i^2}{(cn)^2}\leq c^{-2}, 
    \end{equation*}
    and $\|A - B\|_F^2 = \sum_{\lambda_i\leq cn}\lambda_i^2\leq \delta n^2$, where $A$ is the adjacency matrix of $G$. By Lemma~8.2 of~\cite{graph}, if $\delta$ is sufficiently small with respect to $\eta$, then $G$ admits an $\eta/10$-regular partition into $K$ parts, where $\eta^{-1}\leq K\ll_{c, \eta}1$. Since there are at least $\eta n^3$ induced copies of $P_2$ in $G$, there must be an induced copy of $P_2$ in the auxiliary graph corresponding to the regular partition. By Lemma~8.3 of~\cite{graph}, this implies that $\surp(G)\gg K^{-2}n^2$. Since $\eta = \eta(\epsilon)$, we have $K\ll_{c, \epsilon} 1$. Hence $\surp(G)\gg_{c, \epsilon} n^2$. Since $|\lambda_n|\geq \surp(G)/n$, the desired assertion follows. 
\end{proof}

We are now ready to prove Theorem~\ref{1/4 thm}'. 

\begin{proof}[Proof of Theorem~\ref{1/4 thm}']
    We seek to show that $|\lambda_n|\gg_\epsilon n^{1/4}$. Set $\delta = \conlem{\delta}{dense regpartition reduction prop}(\epsilon) > 0$. Take $c_1 = \clem{dense lem}(\delta) > 0$. We may assume that $|\lambda_n|\leq c_1n^{1/4}$, otherwise the desired claim holds. By Lemma~\ref{dense lem}, 
    \begin{equation*}
        \sum_{\lambda_i\leq c_1n}\lambda_i^2\leq \delta n^2. 
    \end{equation*}
    Hence, by Proposition~\ref{dense regpartition reduction prop}, there exists $c_2 = c_2(\epsilon) > 0$ such that $|\lambda_n| \geq c_2n\geq c_2n^{1/4}$. 
\end{proof}

We note that sufficient proximity to a disjoint union of cliques in particular implies the existence of a rather dense induced subgraph. 

\begin{lem} \label{close to union of cliques implies rather dense subgraph}
    For any $p, q\in (0, 1)$, there exists $\epsilon = \conlem{\epsilon}{close to union of cliques implies rather dense subgraph}(p, q) > 0$ such that the following holds. Let $G$ be an $n$-vertex graph with density at least $p$ which is $\epsilon$-close to a disjoint union of cliques. Then $G$ has an induced subgraph on $pn/2$ vertices with density at least $q$. 
\end{lem}

\begin{proof}
    We take $\epsilon = p^2(1 - q)/100$. We assume that $G$ is not itself a disjoint union of cliques; otherwise, the conclusion is evident. In particular, we may assume that $\epsilon n^2\geq 1$. Hence $n\geq 10$. 
    
    Let $H$ be a disjoint union of cliques on the same vertex set as $G$ which is $\epsilon$-close to $G$. In particular, $\Delta(H)\geq d(H)\geq d(G) - 2\epsilon n\geq pn/2$. Hence there exists a clique $C$ in $H$ of at least $\frac{1}{2}pn + 1$ vertices. Hence
    \begin{equation*}
        e_G(C)\geq e_H(C) - \epsilon n^2\geq e_H(C)\left(1 - \epsilon n^2\middle/\binom{|C|}{2}\right)\geq e_H(C) (1 - \epsilon p^{-2})\geq q\cdot e_H(C), 
    \end{equation*}
    and $G[C]$ is an induced subgraph of $G$ on at least $pn/2$ vertices with density at least $q$. 
\end{proof}

Combining Theorem~\ref{1/4 thm}' and Lemma~\ref{close to union of cliques implies rather dense subgraph} yields the following corollary. 

\begin{cor} \label{rather dense subgraph cor}
    For any $p, q\in (0, 1)$, there exists $c = \clem{rather dense subgraph cor}(p, q) > 0$ such that the following holds. Let $G$ be an $n$-vertex graph with density at least $p$ and least eigenvalue $\lambda_n$. If
    \begin{equation*}
        |\lambda_n|\leq cn^{1/4}
    \end{equation*}
    then $G$ has an induced subgraph on $pn/2$ vertices with density at least $q$. 
\end{cor}

\begin{proof}
    Let $\epsilon = \conlem{\epsilon}{close to union of cliques implies rather dense subgraph}(p, q) > 0$, and let $c = \clem{1/4 thm}(\epsilon)/2 > 0$. Since
    \begin{equation*}
        |\lambda_n|\leq cn^{1/4} < \clem{1/4 thm}(\epsilon)n^{1/4}, 
    \end{equation*}
    by Theorem~\ref{1/4 thm}', $G$ is $\epsilon$-close to a disjoint union of cliques. The desired conclusion now follows from Lemma~\ref{close to union of cliques implies rather dense subgraph}. 
\end{proof}

We finish this section with the proof of Theorem~\ref{5/4 maxcut thm} and Theorem~\ref{5/4 thm}. The argument here mirrors that of the analogous results in~\cite{bridge}. We restate Theorem~\ref{5/4 maxcut thm} in the following equivalent form. 

\begin{custhm}{Theorem}{\ref{5/4 maxcut thm}'}
    For every $\epsilon > 0$, there exists $c = \clem{5/4 maxcut thm}(\epsilon) > 0$ such that if an $n$-vertex graph $G$ is $\epsilon$-far from all disjoint unions of cliques, then 
    \begin{equation*}
        \surp(G)\geq cn^{5/4}. 
    \end{equation*}
\end{custhm}

\begin{proof}
    Without loss of generality, let $\epsilon < 0.1$. Every induced subgraph $H$ of $G$ on at least $(1 - \epsilon/2)n$ vertices is $\epsilon/2$-far from all disjoint unions of cliques. Therefore, by Theorem~\ref{1/4 thm}', the absolute value of the least eigenvalue of $H$ is at least $\clem{1/4 thm}(\epsilon/2)v(H)^{1/4}\geq \clem{1/4 thm}(\epsilon/2)n^{1/4}/2$. Hence by Theorem~\ref{sp bridge thm}, we have
    \begin{equation*}
        \surp(G)\geq \frac{\epsilon n}{2}\cdot \frac{\clem{1/4 thm}(\epsilon/2)n^{1/4}}{2} \gg_\epsilon n^{5/4}, 
    \end{equation*}
    as needed. 
\end{proof}

For the deduction of Theorem~\ref{5/4 thm} from Theorem~\ref{5/4 maxcut thm}', we make use of the following lemmas from~\cite{bridge}. 

\begin{lem}[Lemma~2.1 of~\cite{bridge}] \label{surp and disc lem}
    Let $G$ be a graph on $n$ vertices. Then 
    $$\surp\left(\overline{G}\right)\leq 2\disc^+(G) + \frac{n}{4}.$$ 
\end{lem}

\begin{lem}[Lemma~2.2 of~\cite{bridge}] \label{disc continuity lem}
    Let $G$ and $H$ be graphs of edit distance $s$. Then $$\disc^+(H)\leq \disc^+(G)+s.$$
\end{lem}

\begin{lem}[Lemma~4.3 of~\cite{bridge}] \label{turan lem}
    For every $\epsilon>0$ there exists $\eta = \conlem{\eta}{turan lem}(\epsilon) > 0$ such that if $n$ is sufficiently large and $H$ is a complete multipartite graph on $n$ vertices with $\disc^+(H)\leq \eta n^2$, then $H$ is $\epsilon$-close to a Tur\'an graph.
\end{lem}

We now give a proof of Theorem~\ref{5/4 thm}, which we restate here for convenience. 

\begin{custhm}{Theorem}{\ref{5/4 thm}}
    Let $G$ be an $n$-vertex graph and let $\epsilon > 0$. If $G$ is $\epsilon$-far from all Tur\'an graphs (including the empty graph and the complete graph), then 
    \begin{equation*}
        \disc^+(G)\geq \Omega_\epsilon(n^{5/4}). 
    \end{equation*}
\end{custhm}

\begin{proof}
    Since $\disc^+(G)\geq \epsilon$, we may assume that $n$ is sufficiently large with respect to $\epsilon$. Let $\delta = \min(\conlem{\eta}{turan lem}(\epsilon/2), \epsilon)/2 > 0$. We may assume that $\disc^+(G) < \clem{5/4 maxcut thm}(\delta)n^{5/4}/2 - n/8$; otherwise we are done. By Lemma~\ref{surp and disc lem}, it follows that $\surp\left(\overline{G}\right) < \clem{5/4 maxcut thm}(\delta)n^{5/4}$. Hence by Theorem~\ref{5/4 maxcut thm}', $\overline{G}$ is $\delta$-close to a disjoint union of cliques. Hence $G$ is $\delta$-close to a complete multipartite graph $H$. By Lemma~\ref{disc continuity lem}, it follows that $\disc^+(H)\leq \disc^+(G) + \delta n^2\leq \conlem{\eta}{turan lem}(\epsilon/2)n^2$ when $n$ is sufficiently large with respect to $\epsilon$. Hence by Lemma~\ref{turan lem}, $H$ is $\epsilon/2$-close to a Tur\'an graph. Hence $G$ is $\epsilon$-close to a Tur\'an graph, a contradiction. 
\end{proof}

\section{Triangle count} \label{tri counting sec}

Let $\tri(G)$ denote the number of triangles in a graph $G$. Recall that
\begin{equation*}
    6\tri(G) = \sum_{i = 1}^n\lambda_i^3, 
\end{equation*}
where $\lambda_1\geq \dots\geq \lambda_n$ are the eigenvalues of $G$. 

\begin{prop} \label{triangle counting prop}
    There exist an absolute constant $\clem{triangle counting prop} > 0$ such that the following holds. Let $r\geq 2$ be a positive integer, let $L > 1$ be a positive real number, and let $G$ be an $n$-vertex graph with average degree $d$ and least eigenvalue $\lambda_n$. We assume that $|\lambda_n|\leq L$, $n^{1/2}L^2\leq \clem{triangle counting prop}d$, and 
    \begin{equation} \label{r condition inequality}
        n^{1 - \frac{1}{r}}L^{\frac{2}{r} + \frac{1}{r - 1}}\leq \clem{triangle counting prop}d. 
    \end{equation}
    Then
    \begin{equation} \label{triangle count in terms of r}
        \tri(G)\geq \clem{triangle counting prop}nd^2 L^{-1/r}. 
    \end{equation}
\end{prop}

\begin{proof}
    We take $\clem{triangle counting prop} = 0.01/\Clem{spectral distribution prop}$. Note that $\clem{triangle counting prop}\leq 0.01$ since $\Clem{spectral distribution prop}\geq 1$. 
    
    Since $nL^4\leq \clem{triangle counting prop}^2d^2\leq n^2$, we have $n\geq L^4$. Therefore, by Proposition~\ref{spectral distribution prop}, for any $T\geq n^{1/2}L$ we have
    \begin{equation*}
        \sum_{\lambda_i\leq T}\lambda_i^2\leq \Clem{spectral distribution prop}\left( nL^{\frac{1}{r}}T + n^{3/2}L^2 + n^{2 - \frac{1}{r}}L^{\frac{2}{r} + \frac{1}{r - 1}}\right)\leq \Clem{spectral distribution prop}nL^{\frac{1}{r}}T + 0.02nd. 
    \end{equation*}
    Since the total quadratic mass of all eigenvalues is given by
    \begin{equation*}
        \sum_{i = 1}^n\lambda_i^2 = nd, 
    \end{equation*}
    we have
    \begin{equation*}
        \sum_{\lambda_i > T}\lambda_i^2\geq 0.98nd - \Clem{spectral distribution prop}nL^{1/r}T. 
    \end{equation*}
    Take $T = 0.40\Clem{spectral distribution prop}^{-1}dL^{-1/r} = 40\clem{triangle counting prop}dL^{-1/r}\geq n^{1/2}L$, we have
    \begin{equation*}
        \sum_{\lambda_i > T}\lambda_i^2\geq nd/2. 
    \end{equation*}
    Hence
    \begin{equation*}
        \sum_{\lambda_i\geq 0}\lambda_i^3\geq ndT/2. 
    \end{equation*}
    On the other hand, for the negative eigenvalues, we have
    \begin{equation*}
        \sum_{\lambda_i < 0}|\lambda_i|^3\leq nL^3\leq \clem{triangle counting prop}^2d^2L^{-1}\leq ndT/4, 
    \end{equation*}
    Therefore, 
    \begin{equation*}
        \tri(G) = \frac{1}{6} \sum_{i = 1}^n\lambda_i^3\geq ndT/24\geq \clem{triangle counting prop}nd^2L^{-1/r}, 
    \end{equation*}
    as needed. 
\end{proof}

Whilst the estimate~(\ref{triangle count in terms of r}) strengthens with larger values of $r$, the size of $r$ is constrained by~(\ref{r condition inequality}). Therefore, we now investigate the content of~(\ref{r condition inequality}). In general, for $r\geq 2$ and $c > 0$, an inequality of the form
\begin{equation*}
    n^{1 - \frac{1}{r}}L^{\frac{2}{r} + \frac{1}{r - 1}}\leq cd
\end{equation*}
is equivalent to
\begin{equation*}
    b + \frac{1}{r - 1}\leq (a + b - 2)\frac{1}{r}, 
\end{equation*}
which may further be cast into the quadratic inequality
\begin{equation} \label{r quad ineq in prose}
    b r^2 - (a + 2b - 3)r + (a + b - 2)\leq 0, 
\end{equation}
where $a = \log_L(cd)$ and $b = \log_L \left(\frac{n}{cd}\right)$. In this language, the condition $n^{1/2}L^2\leq cd$ may be written as $a - b \geq 4$, which is equivalent to demanding that $r = 2$ satisfies~(\ref{r quad ineq in prose}). 

\begin{df}
    Let $R\subseteq \mathbb{R}_+^2$ denote the subset of tuples $(a, b)$ satisfying $a > b + 4$. For $(a, b)\in R$, let
    \begin{align*}
        f(a, b) &= \frac{a + 2b - 3 + \sqrt{(a + 2b - 3)^2 - 4b(a + b - 2)}}{2b}\\
        &=\frac{a + 2b - 3 + \sqrt{(a - 3)^2 - 4b}}{2b}. 
    \end{align*}
    be the largest $x\in\mathbb{R}$ satisfying
    \begin{equation} \label{x quad ineq in df}
        b x^2 - (a + 2b - 3)x + (a + b - 2)\leq 0. 
    \end{equation}
\end{df}

Note that for $(a, b)\in R$, $x = 2$ always satisfies~(\ref{x quad ineq in df}) with strict inequality. Therefore, $f(a, b) > 2$ for any $(a, b)\in R$. 

Lemma~\ref{square root lem} records a valid choice of $r$ satisfying~(\ref{r quad ineq in prose}). 

\begin{lem} \label{square root lem} 
    Let $c > 0$, $L > 1$, and $d > 0$ be positive real numbers, and let $n$ be a positive integer such that $n^{1/2}L^2 < cd < n$. Let $a = \log_L(cd)$ and $b = \log_L\left(\frac{n}{cd}\right)$. Then $(a, b)\in R$ and $r = \floor{f(a, b)}$ satisfies 
    \begin{equation*}
        n^{1 - \frac{1}{r}}L^{\frac{2}{r} + \frac{1}{r - 1}}\leq cd. 
    \end{equation*}
\end{lem}

\begin{proof}
    Since $n^{1/2}L^2 < cd$, we have $a - b > 4$, and $(a, b)\in R$. By definition $f(a, b)$ denotes the maximum $x$ satisfying
    \begin{equation} \label{x quad ineq in lem}
        b x^2 - (a + 2b - 3)x + (a + b - 2)\leq 0. 
    \end{equation}
    Note that the set of $x$ satisfying~(\ref{x quad ineq in lem}) forms a closed interval and $x = 2$ satisfies~(\ref{x quad ineq in lem}). Therefore, $x = \floor{f(a, b)}$ must satisfy~(\ref{x quad ineq in lem}), and the desired conclusion follows. 
\end{proof}

The estimate $\tri(G)\gg nd^2L^{-1/r}$ given by Proposition~\ref{triangle counting prop} may be contrasted with the weaker lower bound $\tri(G)\gg nd^2L^{-1}$, which holds whenever $L\leq d^{1/2}/2$. We shall neither prove nor require this weaker estimate, which may be derived directly via trace identities without employing the spectral distribution results established in Section~\ref{spec dist sec}. 

\section{Dense neighbourhoods} \label{dense nbh sec}

The high triangle count established in Section~\ref{tri counting sec} suggests that the neighbourhood of a uniformly randomly chosen vertex in $G$ should be much denser in $G$. Indeed, the induced subgraph on such a neighbourhood is expected to have $\Omega(d^2L^{-1/r})$ edges. This implies a density of $\Omega(L^{-1/r})$ if the size of the neighbourhood is $O(d)$. However, the density might be much lower if the neighbourhood is significantly larger. In other words, to ensure the existence of a sufficiently dense subgraph, it helps if the maximum degree $\Delta\leq O(d)$. 

\begin{lem} \label{find upper bal lem}
    There exists an absolute constant $c = \clem{find upper bal lem} > 0$ such that the following holds. Let $G$ be an $n$-vertex graph with average degree $d$ and least eigenvalue $\lambda_n$ such that $|\lambda_n|\leq d/2$. Then $G$ has an induced subgraph $H$ such that $d(H)\geq cd$ and $\Delta(H)\leq c^{-1}d(H)$. 
\end{lem}

\begin{proof}
    Take $c = 0.01$. We may assume that $e(G) > 0$ since the claim is evident otherwise. Let $G'\sqsubseteq G$ be an induced subgraph maximising $d(G')$ amongst induced subgraphs of $G$. 
    
    Let $S$ denote the set of vertices in $G'$ with degree exceeding $10d(G')$. Note that $|S|\leq v(G')/10$. Let $H = G' - S$ be the induced subgraph of $G$ on $v(G')\backslash S$. By construction, $\Delta(H)\leq 10d(G')$. Therefore, It suffices now to show that $d(H)\geq 10cd(G') = d(G')/10$. For the sake of contradiction, assume otherwise. Therefore, 
    \begin{equation*}
        2e(H) = v(H)d(H) < \frac{v(H)d(G')}{10}\leq \frac{v(G')d(G')}{10} = \frac{1}{5}e(G'). 
    \end{equation*}
    On the other hand, 
    \begin{equation*}
        2e(S) = |S|d(G[S])\leq |S|d(G')\leq \frac{v(G')d(G')}{10} = \frac{1}{5}e(G'). 
    \end{equation*}
    Lastly, we note that $e(S, V(H))\leq e(G')/2 + \surp(G')\leq e(G')/2 + v(G')L/4$. Hence
    \begin{equation*}
        2e(S, v(H))\leq e(G') + \frac{1}{2}v(G')L\leq e(G') + \frac{1}{4}v(G')d(G') = \frac{3}{2}e(G'). 
    \end{equation*}
    Therefore, 
    \begin{align*}
        2e(G')&\leq \frac{1}{5}e(G') + \frac{1}{5}e(G') + \frac{3}{2}e(G') = 1.9e(G')
    \end{align*}
    a contradiction. 
\end{proof}

We now deduce the densification lemma corresponding to the triangle count estimate established by Proposition~\ref{triangle counting prop}. 

\begin{lem} \label{1 step lem}
    There exist an absolute constant $\clem{1 step lem} > 0$ such that the following holds. Let $r\geq 2$ be a positive integer, let $L > 1$ be a positive real number, and let $G$ be an $n$-vertex graph with average degree $d$ and least eigenvalue $\lambda_n$. We assume that $|\lambda_n|\leq L$, $n^{1/2}L^2\leq \clem{1 step lem}d$, and 
    \begin{equation}
        n^{1 - \frac{1}{r}}L^{\frac{2}{r} + \frac{1}{r - 1}}\leq \clem{1 step lem}d. 
    \end{equation}
    Then $G$ has an induced subgraph on at most $d$ vertices with at least
    \begin{equation}
        \clem{1 step lem}d^2 L^{-1/r}
    \end{equation}
    edges. 
\end{lem}

\begin{proof}
    Let $\clem{1 step lem}\in (0, 0.1)$ be fixed later. Since $d\geq \clem{1 step lem}^{-1}L\geq 2L$, by Lemma~\ref{find upper bal lem}, $G$ has an induced subgraph $G'$ such that
    \begin{equation*}
        d(G')\geq \clem{find upper bal lem}d, \qquad \Delta(G')\leq \clem{find upper bal lem}^{-1}d(G'). 
    \end{equation*}
    Let $S$ be the neighbourhood in $G'$ of a uniformly randomly chosen vertex $v$ in $G'$. Note that $\EE{e(G[S])} = 3\tri(G')/v(G')$. Therefore, for some choice of $S$, we have
    \begin{equation*}
        e(G[S]) \geq 3\tri(G')/v(G'). 
    \end{equation*}
    Note that $G'$ and $L$ satisfy the conditions of Proposition~\ref{triangle counting prop} whenever $\clem{1 step lem}$ is sufficiently small with respect to $\clem{triangle counting prop}$. Therefore, 
    \begin{equation*}
        e(G[S]) \geq 3\tri(G')/v(G')\gg d(G')^2L^{-1/r}. 
    \end{equation*}
    Note that $d\geq \clem{1 step lem}^{-1}n^{1/2}L > 10$. Therefore $\binom{\floor{d}}{2}\gg d^2$. By a standard averaging argument, there exists a subset $T\subseteq S$ of at most $d$ vertices such that
    \begin{equation*}
        e(G[T]) \geq \frac{\binom{\floor{d}}{2}}{\binom{|S|}{2}}e(G[S])\gg \frac{d^2}{\Delta(G')^2}d(G')^2L^{-1/r}\gg d^2L^{-1/r}, 
    \end{equation*}
    as needed. 
\end{proof}

\section{Iterative densification} \label{iter densi sec}

We shall separate our analysis based on whether $G$ has a moderately large subgraph of constant density. Formally, we introduce the following definition. 

\begin{df}
    For $p\in (0, 1)$ and $n' > 0$, we call a graph $G$ \emph{$(p, n')$-concentrated} if $G$ has a subgraph on at least $n'$ vertices with density at least $p$; otherwise, we call $G$ \emph{$(p, n')$-dispersed}. 
\end{df}

For a $(p, n')$-concentrated graph $G$, the results of Section~\ref{dense graph sec} apply to its induced subgraph at least $n'$ vertices with density at least $p$. 

\begin{lem} \label{concentrated case lem}
    For any $p, q\in (0, 1)$, there exists $c = \clem{concentrated case lem}(p, q) > 0$ such that the following holds. Let $G$ be a $(p, n')$-concentrated $n$-vertex graph. If the least eigenvalue $\lambda_n$ of $G$ satisfies $|\lambda_n|\leq cn'^{1/4}$, then $G$ has an induced subgraph on $pn'/2$ vertices with density at least $q$. 
\end{lem}

\begin{proof}
    Taking $\clem{concentrated case lem}(p, q) = \clem{rather dense subgraph cor}(p, q) > 0$, the conclusion follows once we apply Corollary~\ref{rather dense subgraph cor} to an induced subgraph of $G$ on at least $n'$ vertices with density at least $p$. 
\end{proof}

For any prescribed density $q\in (0, 1)$, Lemma~\ref{concentrated case lem} allows us to produce an induced subgraph of density at least $q$ from the induced subgraph of density at least $p$. The interest in this densification from one constant density to another lies in the fact that we are able to exert strong controls on the least eigenvalue of graphs with sufficiently high density. For instance, for regular graphs with density above and bounded away from $1/2$, a strong lower bound for the absolute value of the least eigenvalue may be deduced from Lemma~1.6 of~\cite{rst}. In general, for graphs which are not necessarily regular, we record the following, slightly weaker, result from~\cite{graph} applying to graphs with density at least $1 - 10^{-6}$, which suffices for our purposes. 

\begin{lem}[Theorem~9.1(a) of~\cite{graph}] \label{fairly dense lem}
    There exists an absolute constant $\conlem{q}{fairly dense lem}\in (0, 1)$, and for any $\delta > 0$, there exists $N = \conlem{N}{fairly dense lem}(\delta)\in\mathbb{Z}^+$ such that the following holds. Let $G$ be a graph on $n\geq N$ vertices with density at least $q$ and least eigenvalue $\lambda_n$. If $|\lambda_n|\leq n^{1/4 - \delta}$, then $G$ has an induced subgraph on at least $n^{1 - \delta}$ vertices with density at least $1 - |\lambda_n|^2n^{-1 + \delta}$. 
\end{lem}

We now treat iterative densification procedure for $(p, n')$-dispersed graphs. We first introduce the relevant terminology. 

\begin{df}
    Given a $n$-vertex graph $G$ with positive average degree $d$, and a positive real number $L > 1$, let $a_L(G) = \log_{L} (\clem{1 step lem}d)$ and $b_L(G) = \log_L \left(\frac{n}{\clem{1 step lem}d}\right)$. 
\end{df}

We impose a partial order on tuples in $\mathbb{R}^2$. 

\begin{df} \label{tuple po df}
    Let $\succcurlyeq$ be the partial order on $\mathbb{R}^2$ where $(a, b)\succcurlyeq (a', b')$ if $a \geq a'$ and $b\leq b'$. 
\end{df}

\begin{df}
    For a graph $G$ and a positive real number $L$, let $D_L(G)\subseteq \mathbb{R}^2$ be the subset consisting of the tuples $(a_L(H), b_L(H))$ for induced subgraphs $H\sqsubseteq G$ with at least one edge. We endow $D_L(G)\subseteq \mathbb{R}^2$ with the partial order $\succcurlyeq$. We let $\partial D_L(G)$ denote the lower closure in $\mathbb{R}^2$ of $D_L(G)$ with respect to $\succcurlyeq$. 
\end{df}

We combine Lemma~\ref{1 step lem} with the choice of the positive integer $r\geq 2$ given by Lemma~\ref{square root lem} to obtain the following result. 

\begin{lem} \label{1 step tuple lem}
    There exists an absolute constant $\Clem{1 step tuple lem} > 0$ such that the following holds. Let $L > 1$ be a positive real number, and let $G$ be an $n$-vertex graph with least eigenvalue $\lambda_n$ such that $|\lambda_n|\leq L$. If $(a, b)\in D_L(G)\cap R$, then $(a - 1/r - \log_L\Clem{1 step tuple lem}, 1/r + \log_L\Clem{1 step tuple lem})\in \partial D_L(G)$, where $r = \floor{f(a, b)}$. 
\end{lem}

\begin{proof}
    Let $(a, b) = (a_L(H), b_L(H)) = \left(\log_L (\clem{1 step lem}d(H)), \log_L\left(\frac{v(H)}{\clem{1 step lem}d(H)}\right)\right)$ for an induced subgraph $H\sqsubseteq G$ containing at least one edge. 
    
    By Cauchy's interlacing theorem, the absolute value of the least eigenvalue of $H$ is at most $L$. Since $(a, b)\in R$, we have $v(H)^{1/2}L^2 < \clem{1 step lem}d$. By Lemma~\ref{square root lem}, we have $v(H)^{1 - \frac{1}{r}}L^{\frac{2}{r} + \frac{1}{r - 1}}\leq \clem{1 step lem}d$. Hence $H, L$, and $r$ satisfy the conditions of Proposition~\ref{1 step lem}. Therefore, $H$ has an induced subgraph $H'$ on at most $d(H)$ vertices with at least $\clem{1 step lem}d(H)^2L^{-1/r}$ edges. Therefore, $a_L(H')\geq \log_L(2\clem{1 step lem}) + a - 1/r$, and $b_L(H')\leq 1/r - \log_L(2\clem{1 step lem}^2)$. Therefore, taking $\Clem{1 step tuple lem} = \max(\clem{1 step lem}^{-2}, \clem{1 step lem}^{-1})$, we have $(a - 1/r - \log_L\Clem{1 step tuple lem}, 1/r + \log_L\Clem{1 step tuple lem})\in \partial D_L(G)$. 
\end{proof}

We now show that $f$ is strictly increasing with respect to $\succcurlyeq$. 

\begin{lem}
    If $(a, b), (a', b')\in R$ satisfy $(a, b)\succcurlyeq (a', b')$, then $f(a, b)\geq f(a', b')$. Furthermore, equality holds only when $(a, b) = (a', b')$. 
\end{lem}

\begin{proof}
    Let $x = f(a', b')$. By definition, 
    \begin{equation*}
        b'x^2 - (a' + 2b' - 3)x + (a' + b' - 2) = (x - 1)^2 b' - (x - 1)a' + (3x - 2) = 0. 
    \end{equation*}
    Since $x - 1 > 0$, $a\geq a'$ and $b\leq b'$, we have
    \begin{equation} \label{f monotonic proof eq}
        bx^2 - (a + 2b - 3)x + (a + b - 2) = (x - 1)^2 b - (x - 1)a + (3x - 2)\leq 0. 
    \end{equation}
    By the definition of $f(a, b)$, we therefore have $x\leq f(a, b)$ as needed. If $x = f(a, b)$, then equality must hold in~(\ref{f monotonic proof eq}). Hence $a = a'$ and $b = b'$ as needed. 
\end{proof}

\begin{cor} \label{partial to partial cor}
    Let $L > 1$ be a positive real number, and let $G$ be an $n$-vertex graph with least eigenvalue $\lambda_n$ such that $|\lambda_n|\leq L$. If $(a, b)\in \partial D_L(G)\cap R$, then $(a - 1/r - \log_L\Clem{1 step tuple lem}, 1/r + \log_L\Clem{1 step tuple lem})\in \partial D_L(G)$, where $r = \floor{f(a, b)}$. 
\end{cor}

\begin{proof}
    Let $(a', b')\in D_L(G)$ such that $(a', b')\succcurlyeq (a, b)$. Note that $a' - b'\geq a - b\geq 4$, and $b' > 0$ since $(a', b')\in D_L(G)$. Therefore, $(a', b')\in R$. Hence, by Lemma~\ref{1 step tuple lem}, then $(a' - 1/r' - \log_L\Clem{1 step tuple lem}, 1/r' + \log_L\Clem{1 step tuple lem})\in \partial D_L(G)$, where $r' = \floor{f(a', b')}\geq r = \floor{f(a, b)}$. Therefore, $(a' - 1/r' - \log_L\Clem{1 step tuple lem}, 1/r' + \log_L\Clem{1 step tuple lem})\succcurlyeq (a - 1/r - \log_L\Clem{1 step tuple lem}, 1/r + \log_L\Clem{1 step tuple lem})$. Hence $(a - 1/r - \log_L\Clem{1 step tuple lem}, 1/r + \log_L\Clem{1 step tuple lem})\in \partial D_L(G)$, as needed. 
\end{proof}

We now make use of $(p, n')$-dispersiveness. 

\begin{lem} \label{dispersed lemma}
    Let $p\in (0, 1)$, let $n' > 0$, and let $G$ be a $(p, n')$-dispersed graph. Let $L > 1$ be a real number. If $(a, b)\in \partial D_L(G)$ satisfies $a\geq \log_L n'$, then $b\geq \log_L (p^{-1})$. In particular, if $a + b\geq \log_L(p^{-1}n')$, then $b\geq \log_L(p^{-1})$. 
\end{lem}

\begin{proof}
    For the sake of contradiction, let $a\geq \log_L n'$ with $b < \log_L(p^{-1})$. Let $H\sqsubseteq G$ be an induced subgraph of $G$ such that $(a_L(H), b_L(H))\succcurlyeq (a, b)$. Since $\log_L\left(\frac{v(H)}{\clem{1 step lem}d(H)}\right) = b_L(H)\leq b < \log_L(p^{-1})$, the density of $H$ is above $p$. Therefore, as $G$ is $(p, n')$-dispersed, we have $v(H) < n'$. In other words, $a_L(H) + b_L(H) < \log_Ln'$. Therefore, $a\leq a_L(H) < \log_Ln'$, a contradiction. 
\end{proof}

\begin{cor} \label{epsilon densification cor}
    For any $\epsilon\in (0, 1)$, there exists $p = \conlem{p}{epsilon densification cor}(\epsilon)\in (0, 1)$ such that the following holds. Let $n' > 0$. Let $L > 1$ be a positive real number, and let $G$ be a $(p, n')$-dispersed $n$-vertex graph with least eigenvalue $\lambda_n$ such that $|\lambda_n|\leq L$. If $(a, b)\in \partial D_L(G)\cap R$, and $a\geq \log_L(p^{-1}n')$, then $(a - (1 + \epsilon)/r, (1 + \epsilon)/r)\in \partial D_L(G)$, where $r = \floor{f(a, b)}$. 
\end{cor}

\begin{proof}
    By Corollary~\ref{partial to partial cor}, $(a - 1/r - \log_L\Clem{1 step tuple lem}, 1/r + \log_L\Clem{1 step tuple lem})\in \partial D_L(G)$. Therefore, it suffices to show that $(a - (1 + \epsilon)/r, (1 + \epsilon)/r)\preccurlyeq(a - 1/r - \log_L\Clem{1 step tuple lem}, 1/r + \log_L\Clem{1 step tuple lem})$. In other words, it suffices to show that
    \begin{equation} \label{epsilon densification cor target}
        \epsilon/r\geq \log_L \Clem{1 step tuple lem}. 
    \end{equation}
    Since $(a - 1/r - \log_L\Clem{1 step tuple lem}, 1/r + \log_L\Clem{1 step tuple lem})\in \partial D_L(G)$, and 
    \begin{equation*}
        a = (a - 1/r - \log_L\Clem{1 step tuple lem}) + (1/r + \log_L\Clem{1 step tuple lem})\geq \log_L(p^{-1}n'), 
    \end{equation*}
    we have $1/r + \log_L\Clem{1 step tuple lem}\geq \log_L (p^{-1})$ by Lemma~\ref{dispersed lemma}. Therefore,~(\ref{epsilon densification cor target}) follows once we take $p = \Clem{1 step tuple lem}^{-1 - 1/\epsilon}\in (0, 1)$. 
\end{proof}

\section{The recursion} \label{rec sec}

Given a suitable tuple in $\partial D_L(G)$, Corollary~\ref{epsilon densification cor} allows us to produce another tuple in $\partial D_L(G)$. Our strategy will involve iterative applications of Corollary~\ref{epsilon densification cor} starting at a suitable initial tuple in $\partial D_L(G)$. Given an $n$-vertex graph $G$ with average degree $d > 0$, one choice of such a tuple is
\begin{equation*}
    (a_L(G), b_L(G)) = \left(\log_L (\clem{1 step lem}d), \log_L \left(\frac{n}{\clem{1 step lem}d}\right)\right)\in \partial D_L(G). 
\end{equation*}

To study the effect of iterative applications of Corollary~\ref{epsilon densification cor}, we define the following function $g_\epsilon$ which, up to a small modification, records the map given by Corollary~\ref{epsilon densification cor}. Iterative applications of Corollary~\ref{epsilon densification cor} correspond to considering successively higher iterates of $g_\epsilon$. 

\begin{df}
    For $\epsilon\in [0, 1)$, let the function $g_\epsilon: R\rightarrow \mathbb{R}_+^2$ be given by
    \begin{equation*}
        g_\epsilon: (a, b)\mapsto \left(a - \frac{1 + \epsilon}{\ceil{f(a, b)} - 1}, \frac{1 + \epsilon}{\ceil{f(a, b)} - 1}\right). 
    \end{equation*}
    Let 
    \begin{equation*}
        R_{\text{stab}, \epsilon} := \bigcap_{\ell\geq 0}g_\epsilon^{-\ell}(R), 
    \end{equation*}
    where $g_\epsilon^{-\ell}(R)$ denotes the $\ell$th pre-image of $R$ under $g_\epsilon$. In other words, $R_{\text{stab}, \epsilon}\subseteq R$ is the set on which arbitrarily high iterates of $g_\epsilon$ are defined. We abbreviate $g:= g_0$ and $R_{\text{stab}}:= R_{\text{stab}, 0}$. 
\end{df}

Strictly speaking, Corollary~\ref{epsilon densification cor} gives rise to the map $(a, b)\mapsto \left(a - \frac{1 + \epsilon}{\floor{f(a, b)}}, \frac{1 + \epsilon}{\floor{f(a, b)}}\right)$. Taking this map as the definition of $g_\epsilon$ would have sufficed for our results at this stage. However, we have taken $\ceil{f(a, b)} - 1$ in place of $\floor{f(a, b)}$ in our definition of $g_\epsilon$ above in anticipation of passing to the limit $\epsilon\rightarrow 0$, in the process of which a suitable notion of lower-continuity will be essential. 

\begin{obs} \label{g epsilon obs}
    If $(a_1, b_1)\succcurlyeq (a_2, b_2)$ are both elements of $R$, then $g_\epsilon(a_1, b_1)\succcurlyeq g_\epsilon(a_2, b_2)$. In other words, for any fixed $\epsilon$, the function $g_\epsilon$ is order-preserving with respect to $\succcurlyeq$. In particular, since $R$ is an upwards-closed subset of $\mathbb{R}_+^2$ with respect to $\succcurlyeq$, so is $g_\epsilon^{-\ell}(R)$ for any $\ell\geq 0$. Hence their intersection $R_{\text{stab}, \epsilon}$ is an upwards-closed subset of $\mathbb{R}_+^2$ with respect to $\succcurlyeq$. 
\end{obs}

\begin{obs} \label{g epsilon nested obs}
    Given $(a, b)\in R$, and $\epsilon_1 < \epsilon_2$, we have $g_{\epsilon_1}(a, b)\succcurlyeq g_{\epsilon_2}(a, b)$. If $(a, b)\in R_{\text{stab}, \epsilon_2}$, then $g_{\epsilon_1}(a, b)\succcurlyeq g_{\epsilon_2}(a, b)$ must also lie in $R_{\text{stab}, \epsilon_2}$. In other words, $g_{\epsilon_1}$ maps $R_{\text{stab}, \epsilon_2}$ into itself. Hence $R_{\text{stab}, \epsilon_2}\subseteq R_{\text{stab}, \epsilon_1}$. 
\end{obs}

Given $(a, b)\in R_{\text{stab}, \epsilon}$, let $(a_\ell, b_\ell) = g_\epsilon^{(\ell)}(a, b) \in R$. Since $a_{\ell + 1} = a_\ell - b_{\ell + 1} < a_\ell$, $a_\ell$ is a decreasing sequence which is bounded below by 4. Therefore, $a_\ell$ converges to a limit $a_\infty\geq 4$ as $\ell$ tends to infinity. It follows that $b_\ell$ tends to 0 as $\ell$ tends to infinity. This motivates the following definition. 

\begin{df}
    For $\epsilon\in [0, 1)$, and $(a, b)\in R_{\text{stab}, \epsilon}$, let $g_{\infty, \epsilon}(a, b) = a_\infty$, where $g_\epsilon^{(\ell)}(a, b) \rightarrow (a_\infty, 0)$ as $\ell\rightarrow \infty$. Again we abbreviate $g_\infty := g_{\infty, 0}$. 
\end{df}

\begin{prop} \label{abstract recursion result epsilon}
    Let $\epsilon\in (0, 1)$, and let $p = \conlem{p}{epsilon densification cor}(\epsilon)\in (0, 1)$. Let $L > 1$ and let $G$ be an $n$-vertex graph with least eigenvalue $\lambda_n$ such that $|\lambda_n|\leq L$. If $(a, b)\in \partial D_L(G)\cap R_{\text{stab}, \epsilon}$, then $G$ must be $(p, pL^{g_{\infty, \epsilon}(a, b)})$-concentrated. 
\end{prop}

\begin{proof}
    Let $n' = pL^{g_{\infty, \epsilon}(a, b)}$. For the sake of contradiction, we assume that $G$ is $(p, n')$-dispersed. Let $(a_\ell, b_\ell) = g_\epsilon^{(\ell)}(a, b)\in R$. Note that $a_\ell\geq g_{\infty, \epsilon}(a, b) = \log_L(p^{-1}n')$. 
    
    Therefore, by Corollary~\ref{epsilon densification cor}, if $(a_\ell, b_\ell)\in \partial D_L(G)$, then $(a_{\ell + 1}, b_{\ell + 1})\in \partial D_L(G)$. Since $(a_0, b_0) = (a, b)\in \partial D_L(G)$, we may induct on $\ell$ to show that $(a_\ell, b_\ell)\in \partial D_L(G)$ for all $\ell\geq 0$. Recall that $b_\ell\rightarrow 0$ as $\ell\rightarrow \infty$. 
    
    However, for any $\ell\geq 0$, we have $a_\ell\geq g_{\infty, \epsilon}(a, b)\geq \log_L n'$. Hence by Lemma~\ref{dispersed lemma}, $b_\ell\geq \log_L(p^{-1})$. In other words, $b_\ell$ is bounded away from 0, a contradiction to the fact that $b_\ell\rightarrow 0$ as $\ell\rightarrow \infty$. 
\end{proof}

\subsection{Passing to the limiting recursion}

Before we analyse the recursion driven by $g_\epsilon$ further, we first consider the limit as $\epsilon$ tends to zero. We note that for the majority of our purposes, it would suffice to fix a sufficiently small value of $\epsilon > 0$. We take this limit as $\epsilon\rightarrow 0$ here in the interest of removing artificial clutter. The following lemma gives a precise sense in which $R_{\text{stab}, \epsilon}$ approaches $R_{\text{stab}}$ as $\epsilon$ tends to zero. 

\begin{lem} \label{epsilon to 0 interior}
    We have
    \begin{equation*}
        R_{\text{stab}}^\circ = \bigcup_{\epsilon > 0}R_{\text{stab}, \epsilon}^\circ, 
    \end{equation*}
    where $R_{\text{stab}}^\circ$ and $R_{\text{stab}, \epsilon}^\circ$ denote the interiors of $R_{\text{stab}}$ and $R_{\text{stab}, \epsilon}$, respectively. 
\end{lem}

It follows from Observation~\ref{g epsilon nested obs} that $R_{\text{stab}}^\circ \supseteq \cup_{\epsilon > 0}R_{\text{stab}, \epsilon}^\circ$, hence the content of Lemma~\ref{epsilon to 0 interior} lies in the inclusion $R_{\text{stab}}^\circ \subseteq \cup_{\epsilon > 0}R_{\text{stab}, \epsilon}^\circ$. Before we prove Lemma~\ref{epsilon to 0 interior}, we first establish some auxiliary results. Let $R_{\text{stab}, >0} := \cup_{\epsilon > 0}R_{\text{stab}, \epsilon}^\circ$ denote the union on the right-hand side of Lemma~\ref{epsilon to 0 interior}. 

\begin{lem} \label{pull back lem alt}
    Let $V$ be an open subset of $\mathbb{R}_+^2$ which is upward-closed with respect to $\succcurlyeq$. Then $g^{-1}(V)$ is open. 
\end{lem}

\begin{proof}
    For any $(a, b)\in g^{-1}(V)$, we seek to find an open neighbourhood of $(a, b)$ in $g^{-1}(V)$. Since $f: R\rightarrow \mathbb{R}$ is continuous, $U:= \{\ceil{f}\geq \ceil{f(a, b)}\} = \{f > \ceil{f(a, b)} - 1\}$ is open. For $(a', b')\in U$, we have
    \begin{equation*}
        g(a', b')\succcurlyeq h(a', b'):=\left(a' - \frac{1}{\ceil{f(a, b)} - 1}, \frac{1}{\ceil{f(a, b)} - 1}\right)
    \end{equation*}
    Since $h$ is continuous, $h^{-1}(V)$ is open. Hence $W := U\cap h^{-1}(V)$ is open. For any $(a', b')\in W$, we have $h(a', b')\in V$ and hence $g(a', b')\in V$. Therefore, $W\subseteq g^{-1}(V)$ is an open neighbourhood of $(a, b)$ in $g^{-1}(V)$, as needed. 
\end{proof}

\begin{lem} \label{pull back lem}
    For $(a, b)\in R$, if $g(a, b)\in R_{\text{stab}, >0}$, then $(a, b)\in R_{\text{stab}, >0}$. 
\end{lem}

\begin{proof}
    Take $\epsilon_0 > 0$ such that $g(a, b)\in R_{\text{stab}, \epsilon_0}^{\circ }$. In other words, $(a, b)\in g^{-1}(R_{\text{stab}, \epsilon_0}^\circ)$, which is open by Lemma~\ref{pull back lem alt}. Let 
    \begin{equation*}
        h:(a', b')\in R\mapsto \left(a' - \frac{1}{\ceil{f(a, b)} - 1}, \frac{1}{\ceil{f(a, b)} - 1}\right)
    \end{equation*}
    Since $f: R\rightarrow \mathbb{R}$ is continuous, $U:= \{\ceil{f}\geq \ceil{f(a, b)}\} = \{f > \ceil{f(a, b)} - 1\}$ is open. Fix any closed ball $K\subseteq g^{-1}(R_{\text{stab}, \epsilon_0}^\circ)\cap U$ about $(a, b)$. Then $h(K)$ must be a compact subset of $R_{\text{stab}, \epsilon_0}^{\circ}$. Since $g_\epsilon(a', b') - g(a', b') + h(a', b')\rightarrow h(a', b')$ as $\epsilon\rightarrow 0$ uniformly over $(a', b')\in K$, for sufficiently small $\epsilon$, $g_\epsilon(a', b') - g(a', b') + h(a', b')\in R_{\text{stab}, \epsilon_0}^{\circ}$ for all $(a', b')\in K$. Fix such a choice of $\epsilon\in (0, \epsilon_0]$. For any $(a', b')\in K$, we have $g_\epsilon(a', b')\succcurlyeq g_\epsilon(a', b') - g_0(a', b') + h(a', b')\in R_{\text{stab}, \epsilon_0}^{\circ}\subseteq R_{\text{stab}, \epsilon}$, hence $g_\epsilon(a', b')\in R_{\text{stab}, \epsilon}$ and $(a', b')\in R_{\text{stab}, \epsilon}$. Thus $K\subseteq R_{\text{stab}, \epsilon}$. Therefore, $(a, b)\in R_{\text{stab}, \epsilon}^{\circ}\subseteq R_{\text{stab}, >0}$, as needed. 
\end{proof}

\begin{lem} \label{trivial epsilon stab zone}
    For any $\delta\in (0, 0.1)$, let $R_\delta\subseteq \mathbb{R}_+^2$ denote the subset of tuples $(a, b)$ satisfying $b\leq \delta(a - 4 - 10\delta)$. Then for any $\epsilon\in (0, \delta)$, we have $R_\delta\subseteq R_{\text{stab}, \epsilon}$. In particular, $R_\delta^\circ\subseteq R_{\text{stab}, >0}$. 
\end{lem}

\begin{proof}
    Note that $R_\delta\subseteq R$. It suffices to show that $g_\epsilon$ maps $R_\delta$ into itself. Let $(a, b)\in R_\delta$. To show that $g(a, b)\in R_\delta$, we need to prove that
    \begin{equation*}
        \frac{1 + \epsilon}{\ceil{f(a, b)} - 1}\leq \delta \left(a - \frac{1 + \epsilon}{\ceil{f(a, b)} - 1} - 4 - 10\delta\right). 
    \end{equation*}
    Since $(a, b)\in R_\delta$, we have $b\leq \delta \left(a - 4 - 10\delta\right)$. Hence it suffices to show that
    \begin{equation*}
        \frac{1 + \epsilon}{\ceil{f(a, b)} - 1}\leq b - \delta \frac{1 + \epsilon}{\ceil{f(a, b)} - 1}. 
    \end{equation*}
    In other words, we seek to show that
    \begin{equation} \label{trivial epsilon stab zone target}
        \ceil{f(a, b)} - 1\geq \frac{(1 + \epsilon)(1 + \delta)}{b}. 
    \end{equation}
    We have
    \begin{align*}
        f(a, b) &= \frac{a + 2b - 3 + \sqrt{(a - 3)^2 - 4b}}{2b} = \frac{a + 2b - 3 + (a - 3)\sqrt{1 - 4b(a - 3)^{-2}}}{2b}\\
        &\geq \frac{a + 2b - 3 + (a - 3) - 4b(a - 3)^{-1}}{2b} = \frac{a + b - 3}{b} - 2(a - 3)^{-1}\geq \frac{a - 3}{b} - 1. 
    \end{align*}
    Hence $\ceil{f(a, b)} - 1\geq \frac{a - 3}{b} - 2$. Therefore, to establish~(\ref{trivial epsilon stab zone target}), it suffices to show that
    \begin{equation*}
        a - 2b - 3\geq (1 + \epsilon)(1 + \delta).  
    \end{equation*}
    Since $b\leq \delta(a - 4 - 10\delta)$, we have
    \begin{equation*}
        a - 2b - 3\geq a - 3 - 2\delta(a - 4 - 10\delta)\geq 1 + 10\delta\geq (1 + \delta)(1 + \epsilon), 
    \end{equation*}
    as needed. 
\end{proof}

\begin{lem} \label{g infty is bounded away from 4 lem}
    Let $\delta > 0$ and $\epsilon\geq 0$. If $(a, b)\in R_{\text{stab}, \epsilon}$, then $(a + \delta, b)\in R_{\text{stab}, \epsilon}$ and $g_{\infty, \epsilon}(a + \delta, b)\geq 4 + \delta$. 
\end{lem}

\begin{proof}
    Note that $(a + \delta, b)\in R$ and
    \begin{align*}
        g_\epsilon(a + \delta, b) &= \left(a + \delta - \frac{1 + \epsilon}{\ceil{f(a + \delta, b)} - 1}, \frac{1 + \epsilon}{\ceil{f(a + \delta, b)} - 1}\right)\\
        &\succcurlyeq \left(a + \delta - \frac{1 + \epsilon}{\ceil{f(a, b)} - 1}, \frac{1 + \epsilon}{\ceil{f(a, b)} - 1}\right) = g_\epsilon(a, b) + (\delta, 0). 
    \end{align*}
    Therefore, it follows from an induction on $\ell$ that for all $\ell\geq 0$, we have
    \begin{equation*}
        g_\epsilon^{(\ell)}(a + \delta, b) \succcurlyeq g_\epsilon^{(\ell)}(a, b) + (\delta, 0), 
    \end{equation*}
    from which the desired conclusion follows. 
\end{proof}

We now give a proof of Lemma~\ref{epsilon to 0 interior}. 

\begin{proof}[Proof of Lemma~\ref{epsilon to 0 interior}]
    By Observation~\ref{g epsilon nested obs}, $R_{\text{stab}, >0}\subseteq R_{\text{stab}}^\circ$. Therefore, it suffices to show that $R_{\text{stab}}^\circ\subseteq R_{\text{stab}, >0}$. 
    
    Let $(a, b)\in R_{\text{stab}}^\circ$. In particular, for some $\delta > 0$, $(a - \delta, b)\in R_{\text{stab}}^\circ$. Hence by Lemma~\ref{g infty is bounded away from 4 lem}, $g_\infty(a, b)\geq 4 + \delta$. Let $(a_\ell, b_\ell) = g^{(\ell)}(a, b)$, we have $a_\ell\geq 4 + \delta$ for all $\ell\geq 0$. Since $b_\ell\rightarrow 0$ as $\ell\rightarrow\infty$, $(a_\ell, b_\ell)\subseteq R_{\delta/20}^\circ$ whenever $\ell$ is sufficiently large. By Lemma~\ref{trivial epsilon stab zone}, $(a_\ell, b_\ell)\in R_{\text{stab}, >0}$ for sufficiently large $\ell$. Therefore, iteratively applying Lemma~\ref{pull back lem}, it follows that $(a, b)\in R_{\text{stab}, >0}$ as needed. Hence $R_{\text{stab}}^\circ\subseteq R_{\text{stab}, >0}$. 
\end{proof}

Proposition~\ref{abstract recursion result epsilon} and Lemma~\ref{epsilon to 0 interior} combine to give the following result which may be viewed as a form of Proposition~\ref{abstract recursion result epsilon} for the limiting recurrence $g$. 

\begin{prop} \label{abstract recursion result 0}
    For any $(a, b)\in R_{\text{stab}}^{\circ}$, there exist $p = \conlem{p}{abstract recursion result 0}(a, b) > 0$ and $\delta = \conlem{\delta}{abstract recursion result 0}(a, b) > 0$ such that the following holds. Let $L > 1$ and let $G$ be an $n$-vertex graph with least eigenvalue $\lambda_n$ such that $|\lambda_n|\leq L$. If $(a, b)\in \partial D_L(G)$, then $G$ must be $(p, pL^{4 + \delta})$-concentrated. 
\end{prop}

\begin{proof}
    By Lemma~\ref{epsilon to 0 interior}, $(a, b)\in R_{\text{stab}, \epsilon}^\circ$ for some $\epsilon > 0$. In particular, there exists $\delta > 0$ such that $(a - \delta, b)\in R_{\text{stab}, \epsilon}^\circ$. Hence, by Lemma~\ref{g infty is bounded away from 4 lem}, we have $g_{\infty, \epsilon}(a, b)\geq 4 + \delta$. Taking $p = \conlem{p}{abstract recursion result epsilon}(\epsilon) > 0$, the desired assertion follows from Proposition~\ref{abstract recursion result epsilon}. 
\end{proof}

Combining Proposition~\ref{abstract recursion result 0} with Lemmas~\ref{concentrated case lem} and~\ref{fairly dense lem}, we arrive at the following corollary. 

\begin{cor} \label{abstract rec 0 cor}
    For any $(a, b)\in R_{\text{stab}}^{\circ}$, there exists $\delta = \conlem{\delta}{abstract rec 0 cor}(a, b) > 0$ such that the following holds. Let $L > 1$ be sufficiently large with respect to $(a, b)$ and let $G$ be an $n$-vertex graph with least eigenvalue $\lambda_n$ such that $|\lambda_n|\leq L$. If $(a, b)\in \partial D_L(G)$, then $G$ has an induced subgraph $H$ such that $v(H)\geq L^{4 + \delta}$ and $p(H)\geq 1 - L^{-2 - \delta}$. 
\end{cor}

\begin{proof}
    By Proposition~\ref{abstract recursion result 0}, $G$ has an induced subgraph $H'$ such that $v(H')\geq pL^{4 + \delta'}$ and $p(H')\geq p$, where $p = \conlem{p}{abstract recursion result 0}(a, b) > 0$ and $\delta' = \conlem{\delta}{abstract recursion result 0}(a, b) > 0$. Without loss of generality, we assume that $\delta' < 0.1$. When $L$ is sufficiently large with respect to $(a, b)$, Lemma~\ref{concentrated case lem} yields an induced subgraph $H''\sqsubseteq H'$ such that $v(H'')\geq pv(H')/2\geq p^2L^{4 + \delta'}/2$ and $p(H'')\geq \conlem{q}{fairly dense lem}$. Let $N = \conlem{N}{fairly dense lem}(\delta'/10)$. When $L$ is sufficiently large with respect to $(a, b)$, $v(H'')\geq N$. Therefore, by Lemma~\ref{fairly dense lem}, $H''$ has an induced subgraph $H$ such that $v(H)\geq v(H'')^{1 - \delta'/10}\geq L^{4 + \delta'/2}$ and $p(H)\geq 1 - L^2 v(H'')^{-1 + \delta'/10}\geq 1 - L^{-2 - \delta'/2}$ (when $L$ is sufficiently large). Taking $\delta = \delta'/2$ suffices. 
\end{proof}

\subsection{Analysing the limiting recursion}

We now turn our attention to the behaviour of the limiting recursion $g$ and $R_{\text{stab}}$. In particular, we establish an explicit subset of $R_{\text{stab}}$. 

\begin{df}
    Take $R_c\subseteq \mathbb{R}_+^2$ be the subset of tuples $(a, b)$ satisfying $b\leq \frac{1}{2}(a - 4)^2 - \frac{1}{8}(a - 4)^3$. 
\end{df}

Note that if $(a, b)\in R_c$, then $\frac{1}{2}(a - 4)^2 - \frac{1}{8}(a - 4)^3\geq 0$, hence $a\leq 8$ and $b\leq 4$. 

\begin{lem}
    We have $R_c\subseteq R_{\text{stab}}$. 
\end{lem}

\begin{proof}
    Note that $R_c\subseteq R$. It suffices to show that $g$ maps $R_c$ into itself. Let $(a, b)\in R_c$. Note that $b\leq \frac{1}{2}(a - 4)^2 - \frac{1}{8}(a - 4)^3\leq \frac{1}{2}(a - 4)$. Therefore, $(a - 3)^2 - 4b\geq 1 + 2(a - 4) + (a - 4)^2 - 4b\geq 1$. Since $\sqrt{y}\geq \sqrt{x} - (x - y)/2$ for $x\geq y\geq 1$, we have
    \begin{equation*}
        \sqrt{(a - 3)^2 - 4b}\geq a - 3 - 2b. 
    \end{equation*}
    Hence
    \begin{equation*}
        f(a, b) = \frac{a + 2b - 3 + \sqrt{(a - 3)^2 - 4b}}{2b}\geq \frac{a - 3}{b}. 
    \end{equation*}
    Therefore, 
    \begin{equation} \label{r bound in Rc lem}
        r\geq \frac{a - 3}{b} - 1\geq \frac{8(a - 3)}{(a - 4)^2(8 - a)} - 1 = \frac{a^3 - 16a^2 + 88a - 152}{(a - 4)^2(8 - a)}
    \end{equation}
    We seek to show that
    \begin{equation*}
        \frac{1}{r}\leq \frac{1}{2}\left(a - 4 - \frac{1}{r}\right)^2 - \frac{1}{8}\left(a - 4 - \frac{1}{r}\right)^3. 
    \end{equation*}
    We may write this in the form
    \begin{equation} \label{target in Rc lem}
        8\leq \left(a - 4 - \frac{1}{r}\right)^2\left(8r + 1 - ar\right). 
    \end{equation}
    Since the right-hand side of~(\ref{target in Rc lem}) is increasing in $r$, given~(\ref{r bound in Rc lem}), it suffices to show~(\ref{target in Rc lem}) when $r = \frac{a^3 - 16a^2 + 88a - 152}{(a - 4)^2(8 - a)}$. In other words, we seek to show that
    \begin{equation*} \label{target 2 in Rc lem}
        8\leq \left(a - 4 - \frac{(a - 4)^2(8 - a)}{a^3 - 16a^2 + 88a - 152}\right)^2\left(1 + \frac{a^3 - 16a^2 + 88a - 152}{(a - 4)^2}\right). 
    \end{equation*}
    Let $x = a - 4\in [0, 4]$. (\ref{target 2 in Rc lem}) may be written in the form
    \begin{equation*} \label{target 3 in Rc lem}
        8(x^3 - 4x^2 + 8x + 8)^2\leq \left(x^3 - 3x^2 + 4x + 8\right)^2\left(x^3 - 3x^2 + 8x + 8\right). 
    \end{equation*}
    In other words, expanding both sides of~(\ref{target 3 in Rc lem}), we seek to show that
    \begin{equation*} \label{target 4 in Rc lem}
        x^9 - 9x^8 + 43x^7 - 107x^6 + 144x^5 - 24x^4 - 64x^3 + 64x^2\geq 0
    \end{equation*}
    for any $x\in [0, 4]$. This is equivalent to showing that
    \begin{equation*}
        x^7 - 9x^6 + 43x^5 - 107x^4 + 144x^3 - 24x^2 - 64x + 64\geq 0
    \end{equation*}
    for any $x\in [0, 4]$. For $x\in [0, 1]$, this follows since
    \begin{align*}
        &x^7 - 9x^6 + 43x^5 - 107x^4 + 144x^3 - 24x^2 - 64x + 64\\
        &\geq x^7 + 9x^5(1 - x) + 34x^5 + 107(1 - x)x^4 + 37x(1 - x)^2 + 50(1 - x)^2 + (14 - x)\geq 0. 
    \end{align*}
    Therefore, it suffices to show~(\ref{target 4 in Rc lem}) when $x\in [1, 4]$. This is equivalent to showing that
    \begin{equation*}
        x^7 - 2x^6 + 10x^5 + 8x^4 + x^3 + 82x^2 + 60x + 48\geq 0
    \end{equation*}
    for all $x\in [0, 3]$, which follows since
    \begin{equation*}
        x^7 - 2x^6 + 10x^5 + 8x^4 + x^3 + 82x^2 + 60x + 48 = (x - 1)^2x^5 + 9x^5 + 8x^4 + x^3 + 82x^2 + 60x + 48\geq 0. \qedhere
    \end{equation*}
\end{proof}

\section{MaxCut in $K_t$-free graphs} \label{maxcut Kt sec}

This section is dedicated to the proof of Theorem~\ref{Kt lambda n thm} and hence of Theorem~\ref{Kt maxcut thm} which give lower bounds for $|\lambda_n|$ and surplus of $K_t$-free graphs. 

\begin{prop} \label{Kt prop}
    Let $c\in (1/2, 5/8)$ be a real number, and let $t$ be a positive integer. We assume that $\{(a, b)\in \mathbb{R}_+^2\mid a + b + 1\leq c(2a + b)\}\subseteq R_{\text{stab}}^{\circ}$. Let $G$ be $K_t$-free graph with $n$ vertices, $m$ edges, and least eigenvalue $\lambda_n$. Then 
    \begin{equation*}
        n|\lambda_n|\gg_{c, t}m^c. 
    \end{equation*}
\end{prop}

\begin{proof}
    Let $S = \{(a, b)\in \mathbb{R}_+^2\mid a + b + 1\leq c(2a + b)\}\subseteq R_{\text{stab}}^{\circ}$. We may assume that $m > 0$; otherwise, the desired bound holds trivially. Therefore, $|\lambda_n|\geq 1$. Hence we may assume that $n$ and $m$ are sufficiently large with respect to $c$ and $t$. Take $L = (\clem{1 step lem}m)^c/n$. Without loss of generality, we may assume that $L$ is sufficiently large with respect to $c$ and $t$, and $L\geq |\lambda_n|$; otherwise we are done. It follows that $(a_L(G), b_L(G)) = \left(\log_L (\clem{1 step lem}d), \log_L \left(\frac{n}{\clem{1 step lem}d}\right)\right)\in S$. 
    
    We first show that there exists $\epsilon > 0$ such that $S - (\epsilon, 0)\subseteq R_{\text{stab}, \epsilon}$. We show this by decomposing $S$ into three parts, and verifying this property for each part. 

    \begin{itemize}
        \item Note that $S\subseteq ((2c - 1)^{-1}, \infty) \times \mathbb{R}_+ = (4 + \eta, \infty) \times \mathbb{R}_+$, where we take $\eta = (2c - 1)^{-1} - 4 > 0$. As in Lemma~\ref{trivial epsilon stab zone}, for any $\delta > 0$, let $R_\delta\subseteq \mathbb{R}_+^2$ denote the subset of tuples $(a, b)$ satisfying $b\leq \delta(a - 4 - 10\delta)$. Take $\epsilon_1 = \min(\eta, 1)/100$. Let $S_1 := \{(a, b)\in S\mid b < \epsilon_1^2\}\subseteq R_{2\epsilon_1} - (\epsilon_1, 0)\subseteq R_{\text{stab}, \epsilon_1} - (\epsilon_1, 0)$, where the last inclusion follows from Lemma~\ref{trivial epsilon stab zone}. 

        \item Let $S_2 := \{(a, b)\in S\mid a > 10\}$, and let $\epsilon_2' = 0.1$. Therefore, $S_2 - (\epsilon_2', 0)\subseteq R\cap ((9, \infty)\times\mathbb{R}^+)$. For any $(a, b)\in R\cap ((9, \infty)\times\mathbb{R}^+)$, we have
        \begin{equation*}
            g_{\epsilon_2'}(a, b)\succcurlyeq \left(9 - \frac{1 + \epsilon_2'}{2}, \frac{1 + \epsilon_2'}{2}\right)\succcurlyeq \left(6, 0.6\right)\subseteq R_c^\circ\subseteq R_{\text{stab}}^\circ = \bigcup_{\epsilon > 0}R_{\text{stab}, \epsilon}^\circ. 
        \end{equation*}
        Fix $\epsilon_2 \in (0, \epsilon_2')$ such that $\left(6, 0.6\right)\in R_{\text{stab}, \epsilon_2}^\circ$. Therefore, for any $(a, b)\in R\cap ((9, \infty)\times\mathbb{R}^+)$, we have
        \begin{equation*}
            g_{\epsilon_2}(a, b)\succcurlyeq g_{\epsilon_2'}(a, b)\succcurlyeq (6, 0.6)\in R_{\text{stab}, \epsilon_2}, 
        \end{equation*}
        and hence $(a, b)\in R_{\text{stab}, \epsilon_2}$. Therefore, $R\cap ((9, \infty)\times\mathbb{R}^+)\subseteq R_{\text{stab}, \epsilon_2}$. Hence $S_2 - (\epsilon_2, 0)\subseteq S_2 - (\epsilon_2', 0)\subseteq R\cap ((9, \infty)\times\mathbb{R}^+)\subseteq R_{\text{stab}, \epsilon_2}$. 
        
        \item Lastly, let $S_3 = S\backslash (S_1\cup S_2) = \{(a, b)\mid a + b + 1\leq c(2a + b), a\leq 10, b\geq \epsilon_1^2\}$. Note that $S_3$ is compact subset of $R_{\text{stab}}^\circ = \cup_{\epsilon > 0}R_{\text{stab}, \epsilon}^\circ$. Therefore, for some $\epsilon_3' > 0$, $S_3\subseteq R_{\text{stab}, \epsilon_3'}^\circ$. Leveraging the compactness of $S_3$ again, it follows that for a sufficiently small $\epsilon_3\in (0, \epsilon_3')$, we have $S_3 - (\epsilon_3, 0)\subseteq R_{\text{stab}, \epsilon_3}$. 
    \end{itemize}
    Taking $\epsilon = \min(\epsilon_1, \epsilon_2, \epsilon_3) > 0$, it now follows that $S - (\epsilon, 0)\subseteq R_{\text{stab}, \epsilon}$. By Lemma~\ref{g infty is bounded away from 4 lem}, it follows that $S\subseteq R_{\text{stab}, \epsilon}$ and for any $(a, b)\in S$, we have $g_{\infty, \epsilon}(a, b)\geq 4 + \epsilon$. Applying Proposition~\ref{abstract recursion result epsilon} to $(a_L(G), b_L(G)) = \left(\log_L (\clem{1 step lem}d), \log_L \left(\frac{n}{\clem{1 step lem}d}\right)\right)\in S\cap \partial D_L(G)$, it follows that $G$ has an induced subgraph $H$ on at least $pL^{4 + \epsilon}$ vertices with density at least $p$, where $p = \conlem{p}{epsilon densification cor}(\epsilon)\in (0, 1)$. Since $L$ is sufficiently large with respect to $\epsilon$, it follows that $L\leq \clem{concentrated case lem}(p, 1 - t^{-1})v(H)^{1/4}$. Hence Lemma~\ref{concentrated case lem} yields an induced subgraph $H'$ of $H$ with at least $pv(H)/2$ vertices and density at least $1 - t^{-1}$. Therefore, by Tur\'an's theorem, $H'$ must contain a copy of $K_t$, a contradiction. 
\end{proof}

To conclude the proof of Theorem~\ref{Kt lambda n thm}, it suffices to establish that $\{(a, b)\in \mathbb{R}_+^2\mid a + b + 1\leq 0.614(2a + b)\}\subseteq R_{\text{stab}}^{\circ}$. 

\begin{lem} \label{Maxcut leq 6 calculation}
    If $a + b + 1\leq 0.614(2a + b)$ and $a\leq 6$ then $(a, b)\in R_c^{\circ}\subseteq R_{\text{stab}}^{\circ}$. 
\end{lem}

\begin{proof}
    The condition $a + b + 1\leq 0.614(2a + b)$ may be written in the form $0.386b\leq 0.228a - 1$. Assume, for the sake of contradiction, that $(a, b)\notin R_c^\circ$. Therefore, we have $b\geq  \frac{1}{2}(a - 4)^2 - \frac{1}{8}(a - 4)^3$. Hence
    \begin{equation*}
        0.386\left(\frac{1}{2}(a - 4)^2 - \frac{1}{8}(a - 4)^3\right) - 0.228a + 1\leq 0
    \end{equation*}
    for any $a\in [4, 6]$. Taking $x = a - 5\in [-1, 1]$, we have
    \begin{equation*}
        386 x^3 - 386 x^2 - 106 x - 38\geq 0. 
    \end{equation*}
    If $x\geq 0$, then $386 x^3 - 386 x^2 - 106 x - 38 < 386x^2(x - 1)\leq 0$, a contradiction. If $x < 0$, then
    \begin{equation*}
        386x^3 - 386x^2 - 106x - 38 < - 120x^2 - 120x - 30 = -30(2x + 1)^2\leq 0, 
    \end{equation*}
    again yielding a contradiction. 
\end{proof}

\begin{lem} \label{Kt full num estimate}
    We have $\{(a, b)\in \mathbb{R}_+^2\mid a + b + 1\leq 0.614(2a + b)\}\subseteq R_{\text{stab}}^{\circ}$. 
\end{lem}

\begin{proof}
    Let $(a, b)\in \mathbb{R}_+^2$ satisfy $a + b + 1\leq 0.614(2a + b)$. We seek to show that $(a, b)\in R_{\text{stab}}^{\circ}$. When $a\leq 6$, this follows from Lemma~\ref{Maxcut leq 6 calculation}. Therefore, we assume that $a > 6$. Since $(a, b)\in R$, we have $g(a, b)\succcurlyeq (a - 1/2, 1/2)\succcurlyeq (5 + 1/2, 1/2)\in R_c^\circ\subseteq R_{\text{stab}}^{\circ}$. Hence $(a, b)\in R_{\text{stab}}^{\circ}$ by Lemma~\ref{pull back lem}. 
\end{proof}

\begin{proof}[Proof of Theorem~\ref{Kt lambda n thm}]
    By Lemma~\ref{Kt full num estimate}, we have $\{(a, b)\in \mathbb{R}_+^2\mid a + b + 1\leq 0.614(2a + b)\}\subseteq R_{\text{stab}}^{\circ}$. Hence Theorem~\ref{Kt lambda n thm} follows from Proposition~\ref{Kt prop}. 
\end{proof}

\section{Chowla's cosine problem} \label{Chowla cosine sec}

In this section we show Theorem~\ref{chowla main} on Chowla's cosine problem. 

\begin{df}
    Given a finite set $X$ of $n$ positive integers, and a prime $q > 10\max X$, let $\Gamma_{q, X}$ denote the Cayley graph of the group $\mathbb{Z}/q\mathbb{Z}$ generated by $X\cup (-X)$. 
\end{df}

Note that the least eigenvalue $\lambda_q$ of $\Gamma_{q, X}$ is given by
\begin{equation*}
    |\lambda_q| = -\min_{q^{-1}\mathbb{Z}}\widehat{1_{X\cup (-X)}}\leq -\min\widehat{1_{X\cup (-X)}}. 
\end{equation*}

Note that $q$ may be arbitrarily large with respect to $n$. Therefore, whilst
\begin{equation*}
    (a_L(\Gamma_{q, X}), b_L(\Gamma_{q, X})) = \left(\log_L (2\clem{1 step lem}n), \log_L \left(\frac{q}{2\clem{1 step lem}n}\right)\right)\in \partial D_L(\Gamma_{q, X}), 
\end{equation*}
we have no means of guaranteeing that this tuple lies in $R$. Hence, we may not initiate the recursive procedure of Section~\ref{rec sec}. Instead, we consider a different choice of an initial tuple in $\partial D_L(\Gamma_{q, X})$, given by the following lemma. 

\begin{lem} \label{triv densification}
    There exists an absolute constant $\clem{triv densification}\in (0, 1)$ such that the following holds. Let $X$ be a finite set of $n$ positive integers, and let $q > 10\max X$ be a prime. Let $L > 1$ be a positive real number, such that $-\min\widehat{1_{X\cup (-X)}}\leq \clem{triv densification}L$ and $n\geq L^2$, then
    \begin{equation*}
        (\log_L(n) - 1, 1)\in \partial D_L(\Gamma_{q, X}). 
    \end{equation*}
\end{lem}

\begin{proof}
    Take $\clem{triv densification} = \min(1, \clem{1 step lem})/2 > 0$. Let $A$ denote the adjacency matrix of $\Gamma_{q, X}$, and let $\lambda_q$ denote the least eigenvalue of $\Gamma_{q, X}$. Note that $A\succeq -|\lambda_q|I\succeq -\clem{triv densification}L I$. Consider the vector $v = n\cdot 1_{\{0\}} - 2\clem{triv densification}L1_X$. We have
    \begin{equation*}
        4\clem{triv densification}n^2L - 8\clem{triv densification}^2L^2e(X) = -v^TAv \leq \clem{triv densification}Lv^Tv = \clem{triv densification}L(n^2 + 4n\clem{triv densification}^2L^2)\leq 2\clem{triv densification}n^2L. 
    \end{equation*}
    Hence $e(X)\geq \frac{n^2}{4\clem{triv densification}L}$. Therefore, $v(G) = n$ and $\clem{1 step lem}d(G)\geq nL^{-1}$. Hence $(a_L(G), b_L(G))\succcurlyeq \left(\log_L(n) - 1, 1\right)$ and $\left(\log_L(n) - 1, 1\right)\in \partial D_L(\Gamma_{q, X})$. 
\end{proof}

When $\left(\log_L(n) - 1, 1\right)\in R_{\text{stab}}^\circ$, we may leverage the theory developed in the preceding sections together produce an induced subgraph of $\Gamma_{q, X}$ with $\Omega(L^4)$ vertices and density $1 - O(L^{-2})$. The following proposition gives a precise form of this result. 

\begin{prop} \label{chowla dense subgraph prop}
    Let $c$ be a positive real number such that $(c, 1)\in R_{\text{stab}}^\circ$. Let $X$ be a set of $n$ positive integers. If $n$ is sufficiently large with respect to $c$ and
    \begin{equation*}
        -\min \widehat{1_{X\cup (-X)}} < \clem{triv densification}n^{1/(c + 1)}, 
    \end{equation*}
    then for any prime $q > 10\max X$, $\Gamma_{q, X}$ has an induced subgraph $G$ such that $v(G)\geq n^{4/(c + 1)}$ and $\Delta(\overline{G})\leq n^{2/(c + 1)}$. 
\end{prop}

\begin{proof}
    We implicitly assume throughout that $n$ is sufficiently large with respect to $c$. Note that since $(c, 1)\in R_{\text{stab}}^\circ\subseteq R$, we have $c\geq 5$. Let $L = n^{1/(c + 1)}$. Note that $L > 1$, $\clem{triv densification}L\geq -\min \widehat{1_{X\cup (-X)}}$, and $n\geq L^2$. Therefore, by Lemma~\ref{triv densification}, $(c, 1)\in \partial D_L(\Gamma_{q, X})$. 

    Note that the absolute value of the least eigenvalue of $\Gamma_{q, X}$ is bounded above by $L$. Applying Corollary~\ref{abstract rec 0 cor} to $\Gamma_{q, X}$, $L$ and $(c, 1)$, it follows that $\Gamma_{q, X}$ has an induced subgraph $G_1$ with at least $L^{4 + \conlem{\delta}{abstract rec 0 cor}}\geq 2L^4$ vertices and density at least $1 - L^{-2 - \conlem{\delta}{abstract rec 0 cor}}\geq 1 - L^{-2}/10$. By a standard averaging argument, $G_1$ has an induced subgraph $G_2$ with $\ceil{2L^4}$ vertices and density at least $1 - L^{-2}/10$. $\overline{G_2}$ has at most $L^4$ vertices of degree great than $L^2$. Removing these vertices from $G_2$ yields an induced subgraph $G$ of $G_2$ on at least $L^4$ vertices with $\Delta(\overline{G})\leq L^2$, as needed. 
\end{proof}

Tur\'an's theorem may be applied to this dense subgraph given by Proposition~\ref{chowla dense subgraph prop} to yield a complete subgraph on $\Omega(L^2)$ vertices. Up to different parameters, this was the approach considered by~\cite{graph} in showing that $\mathfrak{m}(n)\gg n^{1/10}$. However, to give the sharper bound of Theorem~\ref{chowla main}, this application of Tur\'an's theorem turns out to be too costly, instead we need to work directly with this dense induced subgraph. 

We start with the following lemma showing that any other vertex must be either adjacent to almost all vertices of this dense induced subgraph or almost none of them. 

\begin{lem} \label{half graph lem}
    Let $G$ be a graph on $n > 1$ vertices with least eigenvalue $\lambda_n$. For any vertex $i\in V(G)$, if $8\Delta(\overline{G - i})|\lambda_n|\leq n - 1$, then
    \begin{equation*}
        \min(\deg_G(i), \deg_{\overline{G}}(i))\leq 4|\lambda_n|^2. 
    \end{equation*}
\end{lem}

\begin{proof}
    We may assume that $e(G) > 0$ (and hence $\lambda_n < 0$); otherwise, the desired assertion is trivial. 
    
    Let $A$ denote the adjacency matrix of $G$. Note that $A\succeq -|\lambda_n|I$. Let $B$ denote the neighbourhood of $i$ in $G$ and let $C$ denote the neighbourhood of $i$ in $\overline{G}$. For the sake of contradiction, assume that $\min(|B|, |C|) > 4|\lambda_n|^2$. 
    
    For some $x\in\mathbb{R}$ to be fixed later, consider the vector $v\in\mathbb{R}^{V(G)}$ given by $v_i = x$, $v_j = -1/|B|$ for $j\in B$, and $v_j = 1/|C|$ for $j\in C$. Note that
    \begin{align*}
        - v^TA_Gv &= 2x - 2\cdot\frac{e_G(B)}{|B|^2} - 2\cdot\frac{e_G(C)}{|C|^2} + 2\cdot\frac{e_G(B, C)}{|B||C|}\geq 2x - 1 - 1 + 2\cdot\frac{|B||C| - e_{\overline{G}}(B, C)}{|B||C|}\\
        &= 2x - \frac{2e_{\overline{G}}(B, C)}{|B||C|}\geq 2x - \frac{2\min(|B|, |C|)\Delta(\overline{G - i})}{|B||C|}= 2x - \frac{2\Delta(\overline{G - i})}{\max(|B|, |C|)}\geq 2x - \frac{4\Delta(\overline{G - i})}{n - 1}\\
        &\geq 2x - \frac{1}{2|\lambda_n|}
    \end{align*}
    and
    \begin{equation*}
        v^Tv = x^2 + \frac{1}{|B|} + \frac{1}{|C|}\leq x^2 + \frac{2}{\min(|B|, |C|)} < x^2 + \frac{1}{2|\lambda_n|^2}. 
    \end{equation*}
    Since $- v^TA_Gv\leq |\lambda_n|v^Tv$, we must have
    \begin{equation*}
        |\lambda_n|x^2 - 2x + \frac{1}{|\lambda_n|} > 0. 
    \end{equation*}
    for any $x\in\mathbb{R}$. Taking $x = |\lambda_n|^{-1}$ gives the desired contradiction. 
\end{proof}

\begin{lem} \label{very dense no subgraph passage ln bound lem}
    There exists an absolute constant $\clem{very dense no subgraph passage ln bound lem}\in (0, 1)$ such that the following holds. Let $G$ be an $n$-vertex graph with least eigenvalue $\lambda_n$. If $\Delta(\overline{G})\leq \clem{very dense no subgraph passage ln bound lem}n^{2/3}$, then 
    \begin{equation*}
        \Delta(\overline{G})\leq \clem{very dense no subgraph passage ln bound lem}^{-1}|\lambda_n|^2. 
    \end{equation*}
\end{lem}

\begin{proof}
    Let $A, \overline{A}$ denote the adjacency matrices of $G, \overline{G}$, respectively. Note that $\overline{A} = J - I - A\preceq J + (|\lambda_n| - 1)I$. Let $Y = J + (|\lambda_n| - 1)I - \overline{A}\succeq 0$. Let $\overline{\Delta} = \Delta(\overline{G})$ denote the maximum degree of $\overline{G}$. 
    
    Let $i$ be a vertex such that $\deg_{\overline{G}} (i) = \overline{\Delta}$ and let $B$ be the neighbourhood of $i$ in $\overline{G}$. Let $v = n\cdot 1_{\{i\}} - 1$ and let $u = n\cdot 1_B - \overline{\Delta}$. We have
    \begin{align*}
        v^T\overline{A}v &= -2n\overline{\Delta} + 2e(\overline{G})\geq -2n\overline{\Delta}, \\
        v^T\overline{A}u &= n^2\overline{\Delta} - n\overline{\Delta}^2 - n\sum_{j\in B}\deg_{\overline{G}}(j) + 2\overline{\Delta}e(\overline{G})\geq n\overline{\Delta}(n - 2\overline{\Delta})
        \\
        u^T\overline{A}u &= 2n^2e_{\overline{G}}(B) - 2n\overline{\Delta}\sum_{j\in B}\deg_{\overline{G}}(j) + 2\overline{\Delta}^2e(\overline{G})\geq -2n\overline{\Delta}^3, 
    \end{align*}
    and $v^Tv = n(n - 1), v^Tu = -n\overline{\Delta}, u^Tu = n\overline{\Delta}(n - \overline{\Delta})$. Therefore, since $Jv = Ju = 0$, we have
    \begin{align*}
        v^TYv &\leq n^2(|\lambda_n| - 1) + 2n\overline{\Delta}\leq n^2|\lambda_n|, \\
        v^TYu &\leq - n\overline{\Delta}(n - 2\overline{\Delta}), \\
        u^TYu &\leq n^2\overline{\Delta}(|\lambda_n| - 1) + 2n\overline{\Delta}^3\leq n^2\overline{\Delta}|\lambda_n| + 2n\overline{\Delta}^3. 
    \end{align*}
    Let $w = \overline{\Delta}^{1/2}v + u$. Since $Y\succeq 0$, we have
    \begin{equation*}
        0\leq w^TYw\leq 2n^2\overline{\Delta}|\lambda_n| + 2n\overline{\Delta}^3 - 2n\overline{\Delta}^{3/2}(n - 2\overline{\Delta}). 
    \end{equation*}
    Hence
    \begin{equation*}
        n|\lambda_n|\geq \overline{\Delta}^{1/2}(n - 2\overline{\Delta}) - \overline{\Delta}^2\geq (1 - 2\clem{very dense no subgraph passage ln bound lem} - \clem{very dense no subgraph passage ln bound lem}^{3/2})\overline{\Delta}^{1/2}n. 
    \end{equation*}
    Taking $\clem{very dense no subgraph passage ln bound lem} = 0.1$ yields the desired bound. 
\end{proof}

\begin{prop} \label{chowla no dense subgraph prop}
    Let $L > 1$ be a real number, let $X$ be a set of $n$ positive integers, and let $q > 10\max(\max X, L^3)$ be a prime number. We assume that $L$ is sufficiently large (in an absolute sense), 
    \begin{equation*}
        -\min \widehat{1_{X\cup (-X)}}\leq L, 
    \end{equation*}
    and $\Gamma_{q, X}$ has an induced subgraph $G$ such that $v(G)\geq L^4$ and $\Delta(\overline{G})\leq L^2$. Then $n\gg L^7$. 
\end{prop}

\begin{proof}
    Let $H$ maximise $v(H)$ amongst induced subgraphs of $\Gamma_{q, X}$ satisfying $\Delta(\overline{H})\leq \clem{very dense no subgraph passage ln bound lem}^{-1}L^2$. Note that, in particular, $v(H)\geq v(G)\geq L^4$. 

    We first show that if $i\in V(\Gamma_{q, X})\backslash V(H)$ is a vertex not in $H$, then $e(\{i\}, H)\leq 4L^2$. Assume otherwise. By Lemma~\ref{half graph lem}, we must have $e_{\overline{\Gamma_{q, X}}}(\{i\}, H)\leq 4L^2$. Let $H'$ be the induced subgraph of $\Gamma_{q, X}$ on $V(H)\cup \{i\}$. Therefore, $\Delta(\overline{H'})\leq \max(\clem{very dense no subgraph passage ln bound lem}^{-1}L^2 + 1, 4L^2)\leq (\clem{very dense no subgraph passage ln bound lem}^{-1} + 4)L^2\leq \clem{very dense no subgraph passage ln bound lem}L^{8/3}\leq \clem{very dense no subgraph passage ln bound lem}v(H')^{2/3}$. Hence by Lemma~\ref{very dense no subgraph passage ln bound lem}, we have
    \begin{equation*}
        \Delta(\overline{H'})\leq \clem{very dense no subgraph passage ln bound lem}^{-1}L^2, 
    \end{equation*}
    a contradiction to the maximality of $v(H)$. 

    Therefore, $e(\{i\}, H)\leq 4L^2$ for any $i\in V(\Gamma_{q, X})\backslash V(H)$. Since $q\geq 10\max X$, for any edge $ij\in E(\Gamma_{q, X})$, we may unambiguously define $|i - j|\in X$. We consider the function
    \begin{align*}
        \operatorname{diff}: E(H)&\rightarrow X, \\
        ij&\mapsto |i - j|. 
    \end{align*}
    By the pigeonhole principle, there exists some $t\in X$ such that $|\operatorname{diff}^{-1}(\{t\})|\geq \frac{e(H)}{|X|}\gg L^8/n$. In other words, $|V(H)\cap (V(H) + t)|\gg L^8/n$. 
    
    If $V(H) = V(H) + t$, then $V(H) = V(\Gamma_{q, X})$. Hence $\Delta(\overline{\Gamma_{q, X}}) = q - 2n - 1\leq \clem{very dense no subgraph passage ln bound lem}^{-1}L^2$. Thus $q\ll L^2$, a contradiction. Hence we may assume that $V(H)\neq V(H) + t$. 
    
    Let $S = V(H)\cap (V(H) + t)$, and let $T$ be the symmetric difference of $V(H)$ and $(V(H) + t)$. Take any $i\in (V(H) + t)\backslash V(H)$. We have
    \begin{equation*}
        |S| = e(\{i\}, S) + e_{\overline{\Gamma_{q, X}}}(\{i\}, S)\leq e(\{i\}, V(H)) + \deg_{\overline{H}}(i - t)\leq 4L^2 + \clem{very dense no subgraph passage ln bound lem}^{-1}L^2\ll L^2. 
    \end{equation*}
    In particular, when $L$ is sufficiently large, we have $|S|\leq v(H)/2$ and $|T|\geq v(H)$. 
    
    Let $A$ denote the adjacency matrix of $\Gamma_{q, X}$. Consider the vector $v = |T|\cdot 1_S - |S|\cdot 1_T$. We have
    \begin{align*}
        -v^TAv &= - 2|T|^2e(S) - 2|S|^2e(T) + 2|S||T|e(S, T)\\
        &= 2|T|^2e_{\overline{\Gamma_{q, X}}}(S) + 2|S|^2e_{\overline{\Gamma_{q, X}}}(T) - 2|S||T|e_{\overline{\Gamma_{q, X}}}(S, T) + |S|^2|T| + |S||T|^2\\
        &\geq 2|S|^2e_{\overline{\Gamma_{q, X}}}(T) - 2|S||T|e_{\overline{\Gamma_{q, X}}}(S, T)\\
        &\geq 2|S|^2e_{\overline{\Gamma_{q, X}}}(V(H)\backslash (V(H) + t), (V(H) + t)\backslash V(H)) - 4\clem{very dense no subgraph passage ln bound lem}^{-1}|S|^2|T|L^2\\
        &= |S|^2|T|^2/2 - 2|S|^2e(V(H)\backslash (V(H) + t), (V(H) + t)\backslash V(H)) - 4\clem{very dense no subgraph passage ln bound lem}^{-1}|S|^2|T|L^2\\
        &\geq |S|^2|T|^2/2 - 4|S|^2|T|L^2 - 4\clem{very dense no subgraph passage ln bound lem}^{-1}|S|^2|T|L^2
    \end{align*}
    and $v^Tv = |S|^2|T| + |S||T|^2$. Since $-v^TAv\leq Lv^Tv$, we have
    \begin{equation*}
        |T|/2 - 4L^2 - 4\clem{very dense no subgraph passage ln bound lem}^{-1}L^2\leq L + L|T|/|S|. 
    \end{equation*}
    Hence
    \begin{equation*}
        v(H)/2 - 8\clem{very dense no subgraph passage ln bound lem}^{-1}L^2\leq 4Lv(H)/|S|. 
    \end{equation*}
    When $L$ is sufficiently large, $8\clem{very dense no subgraph passage ln bound lem}^{-1}L^2\leq L^4/4\leq v(H)/4$. Hence $|S|\ll L$. However, we have established that $|S|\gg L^8/n$. Hence $n\gg L^7$ as asserted. 
\end{proof}

\begin{prop} \label{chowla prop}
    If $(c, 1)\in R_{\text{stab}}^\circ$ and $c < 6$, then 
    \begin{equation*}
        \mathfrak{m}(n)\gg n^{1/(c + 1)}. 
    \end{equation*}
\end{prop}

\begin{proof}
    We may assume that $n$ is sufficiently large with respect to $c$. Let $L = n^{1/(c + 1)}$. Let $X$ be any set of $n$ positive integers. We seek to show that 
    \begin{equation*}
        -\min \widehat{1_{X\cup (-X)}}\geq \clem{triv densification}n^{1/(c + 1)}. 
    \end{equation*}
    Assume otherwise. Let $q > \max(10\max X, L^3)$ be a prime number. By Proposition~\ref{chowla dense subgraph prop}, $\Gamma_{q, X}$ has an induced subgraph $G$ such that $v(G)\geq L^4$ and $\Delta(\overline{G})\leq L^2$. By Proposition~\ref{chowla no dense subgraph prop}, we have $n = L^{c + 1}\gg L^7$, a contradiction. 
\end{proof}

To deduce Theorem~\ref{chowla main} from Proposition~\ref{chowla prop}, it remains to show that $(5.43, 1)\in R_{\text{stab}}^\circ$. 

\begin{lem}
    We have $(5.43, 1)\in R_{\text{stab}}^\circ$. 
\end{lem}

\begin{proof}
    We have $(5.43, 1)\in R$, $g(5.43, 1) = (4.93, 1/2)\in R$, $g^{(2)}(5.43, 1) = (4.68, \frac{1}{4})\in R$, $g^{(3)}(5.43, 1) = (\frac{794}{175}, \frac{1}{7})\in R$, $g^{(4)}(5.43, 1) = (\frac{8559}{1925}, \frac{1}{11})\in R$, and $g^{(5)}(5.43, 1) = (\frac{135019}{30800}, \frac{1}{16})\in  R_c^{\circ}\subseteq R_{\text{stab}}^\circ$. Hence by Lemma~\ref{pull back lem alt}, $(5.43, 1)\in R_{\text{stab}}^\circ$ as asserted. 
\end{proof}

\section{Polynomial proximity to a disjoint union of cliques} \label{poly sec}

In this section we give a proof of Theorem~\ref{1/4 poly thm}, which we restate here for convenience. Along with input from the theory we developed earlier, we loosely follow the proof of Theorem~\ref{1/6 poly thm} in~\cite{graph}. 

\begin{custhm}{Theorem}{\ref{1/4 poly thm}}
    For any $\gamma\in (0, 1/4)$, there exists $\alpha > 0$ such that the following holds when $n$ is sufficiently large with respect to $\gamma$. If the least eigenvalue $\lambda_n$ of an $n$-vertex graph $G$ satisfies $|\lambda_n|\leq n^\gamma$, then $G$ is $n^{-\alpha}$-close to a disjoint union of cliques. 
\end{custhm}

For the sake of clarity, we do not track or optimise the dependency of $\alpha$ on $\gamma$. However, a more careful analysis of our argument shows that we can in fact take $\alpha = (16 + o(1))(1/4 - \gamma)^2$ as $\gamma\rightarrow 1/4$. 

We first record the following result asserting the existence of a large clique. 

\begin{prop} \label{1/4 poly clique}
    Let $(a, b)\in R_{\text{stab}}^{\circ}$, and let $L > 1$ be sufficiently large with respect to $(a, b)$. If $G$ is an $n$-vertex graph with least eigenvalue $\lambda_n$ such that $|\lambda_n|\leq L$ and $(a, b)\in \partial D_L(G)$, then $G$ contains a clique on $L^{2 + \conlem{\delta}{abstract rec 0 cor}(a, b)}$ vertices. 
\end{prop}

\begin{proof}
    By Corollary~\ref{abstract rec 0 cor}, $G$ contains an induced subgraph $H$ such that $v(H)\geq L^{4 + \conlem{\delta}{abstract rec 0 cor}(a, b)}$ and $p(H)\geq 1 - L^{-2 - \conlem{\delta}{abstract rec 0 cor}(a, b)}$. The conclusion follows by applying Tur\'an's theorem to $H$. 
\end{proof}

We record a handful of simple lemmas from~\cite{graph}. 

\begin{lem}[Simple consequence of Claim~5.1 and Lemma~11.1 of~\cite{graph}] \label{poly dense Y lem}
    Let $G$ be a graph on $n$ vertices with least eigenvalue $\lambda_n$. Let $X\cup Y$ be a partition of $V(G)$, then $n^3|\lambda_n|\geq e(X, Y)^2 - 4n^2e(Y)$. 
\end{lem}

\begin{lem}[Lemma~11.4 of~\cite{graph}] \label{poly dichotomy lem}
    There exists an absolute constant $C = \Clem{poly dichotomy lem} \geq 1$ such that the following holds. Let $G$ be an $n$-vertex graph with least eigenvalue $\lambda_n$ and let $X, Y\subseteq V(G)$ disjoint subsets of vertices of the same size such that $G[X]$ and $G[Y]$ are complete. Then either $e(X, Y)\leq C |\lambda_n|^2|X|$ or $e_{\overline{G}}(X, Y)\leq C|\lambda_n|^2|X|$. 
\end{lem}

\begin{lem}[Simple consequence of Claim~5.1 and Lemma~8.3 of~\cite{graph}] \label{poly P2 lem}
    Let $G$ be an $n$-vertex graph with least eigenvalue $\lambda_n$ and let $X, Y, Z\subseteq V(G)$ disjoint subsets of vertices of the same size. Let $\delta > 0$ be a real number such that
    $e_{\overline{G}}(X, Y)\leq \delta|X|^2$, $e_{\overline{G}}(X, Z)\leq \delta|X|^2$ and $e(Y, Z)\leq \delta|X|^2$. Then $|\lambda_n|\geq (1/3 - 4\delta)|X|$. 
\end{lem}

\begin{proof}[Proof of Theorem~\ref{1/4 poly thm}]
    Without loss of generality, we assume that $\gamma\in (1/5, 1/4)$. We implicitly assume throughout that $n$ is sufficiently large with respect to $\gamma$. Fix $\beta > 0$ sufficiently small with respect to $\gamma$ so that $(\gamma^{-1}(1 - 2\beta), 2\gamma^{-1}\beta)\in R_c^\circ\subseteq R_{\text{stab}}^\circ$. Let $\delta = \min\left(\beta, \conlem{\delta}{abstract rec 0 cor}\left(\gamma^{-1}(1 - 2\beta), 2\gamma^{-1}\beta\right)\right) > 0$. We take $\alpha = \gamma\delta/2$. It suffices to show that $G$ is $O(n^{-\gamma\delta})$-close to a disjoint union of cliques. 

    Let $V_1, \dots, V_k$ denote a maximal collection of disjoint subsets of $n^{2\gamma + \gamma\delta}$ vertices each such that $G[V_i]$ is complete for each $i\in [k]$. Let $V = \cup_i V_i$. 
    
    We seek to show that $e(G) - e(V)\leq 4\clem{1 step lem}^{-1/2}n^{2 - \beta}$. Assume otherwise. We now show that $e(V^c)\geq \clem{1 step lem}^{-1} n^{2 - 2\beta}$. Applying Lemma~\ref{poly dense Y lem} with $X = V$, $Y = V^c$, we have
    \begin{equation*}
        n^{3 + \gamma}\geq e(V, V^c)^2 - 4n^2e(V^c). 
    \end{equation*}
    Therefore, if $e(V^c) < \clem{1 step lem}^{-1} n^{2 - 2\beta}$, we have
    \begin{equation*}
        e(V, V^c)^2\leq 5\clem{1 step lem}^{-1}n^{4 - 2\beta}. 
    \end{equation*}
    Hence $e(V, V^c)\leq 3\clem{1 step lem}^{-1/2}n^{2 - \beta}$ and $e(G) - e(V)\leq (3\clem{1 step lem}^{-1/2} + o(1))n^{2 - \beta}$, a contradiction. Hence $e(V^c)\geq \clem{1 step lem}^{-1} n^{2 - 2\beta}$. Let $H:= G - V$ be the induced subgraph of $G$ on $V^c$ and let $L = n^\gamma$. Since $\clem{1 step lem}d(H)\geq \clem{1 step lem}e(H)/n\geq n^{1 - 2\beta}$ and $\frac{v(H)}{\clem{1 step lem}d(H)}\leq n^{2\beta}$, it follows that $\left(a_L(H), b_L(H)\right)\succcurlyeq \left(\gamma^{-1}(1 - 2\beta), 2\gamma^{-1}\beta\right)$. Hence by Proposition~\ref{1/4 poly clique}, $H$ contains a clique on $L^{2 + \delta} = n^{2\gamma + \gamma\delta}$ vertices, a contradiction to the maximality of the collection $V_1, \dots, V_k$. 

    Therefore, we must have $e(G) - e(V)\leq 4\clem{1 step lem}^{-1/2}n^{2 - \gamma\delta}$. It suffices to show that the edit distance between $G:= G[V]$ and a disjoint union of cliques is at most $O_{\gamma}(n^{2 - \gamma\delta})$. Note that $|V_1| = \dots = |V_k| = n^{2\gamma + \gamma\delta}\geq 100\Clem{poly dichotomy lem}L^2$. By Lemma~\ref{poly dichotomy lem}, for each pair of $i\neq j\in [k]$, either $e(V_i, V_j)\leq \Clem{poly dichotomy lem}L^2|V_1|\leq 0.01 |V_1|^2$ or $e_{\overline{G}}(V_i, V_j)\leq 0.01 |V_1|^2$. Construct an auxiliary graph $\Gamma$ on vertex set $[k]$ where $ij\in E(\Gamma)$ if $e_{\overline{G}}(V_i, V_j)\leq 0.01 |V_1|^2$. If there is an induced copy of $P_2$ (the path on three vertices and two edges) in $\Gamma$, then $|\lambda_n|\geq (1/3 - 0.04)|V_1|\gg n^{2\gamma}$ by Lemma~\ref{poly P2 lem}, a contradiction. Hence $\Gamma$ is a disjoint union of cliques. Let $G'$ denote the graph on the vertex set $V$, where $G[V_i]$ is complete for each $i\in [k]$ and $G[V_i, V_j]$ is complete if $ij\in \Gamma$ and empty otherwise. $G'$ is a disjoint union of cliques and the edit distance between $G$ and $G'$ is at most
    \begin{equation*}
        k^2\cdot \Clem{poly dichotomy lem}L^2|V_1|\ll n^{2 - \gamma\delta}, 
    \end{equation*}
    as desired. 
\end{proof}

\textbf{AI usage declaration.} The mathematical content of this paper is human-generated. ChatGPT Pro is used to write Python scripts for initial numerical experiments, on which the mathematical content does not depend. ChatGPT Pro is also used to polish an initial version of this manuscript. 

\textbf{Acknowledgements.} The author would like to thank Oliver Janzer for advice and guidance on this paper, and Timothy Gowers for helpful suggestions on an initial version submitted as a part III essay. The author would also like to thank Zhihan Jin for pointing out an error in a previous version of the proof of Proposition~\ref{key} and suggesting an AI-generated correction. 

\bibliographystyle{abbrv}
\bibliography{mybib}

\end{document}